\documentclass[reqno]{amsart}
\usepackage{amsfonts}

\newtheorem{theorem}{Theorem}
\theoremstyle{plain}

\newtheorem{corollary}{Corollary}

\newtheorem{example}{Example}

\newtheorem{lemma}{Lemma}

\newtheorem{proposition}{Proposition}
\newtheorem{remark}{Remark}

\numberwithin{equation}{section}
\numberwithin{theorem}{section}
\numberwithin{algorithm}{section}
\numberwithin{axiom}{section}
\numberwithin{case}{section}
\numberwithin{claim}{section}
\numberwithin{conclusion}{section}
\numberwithin{condition}{section}
\numberwithin{conjecture}{section}
\numberwithin{corollary}{section}
\numberwithin{criterion}{section}
\numberwithin{definition}{section}
\numberwithin{example}{section}
\numberwithin{exercise}{section}
\numberwithin{lemma}{section}
\numberwithin{notation}{section}
\numberwithin{problem}{section}
\numberwithin{proposition}{section}
\numberwithin{remark}{section}
\numberwithin{solution}{section}

\input{tcilatex}

\begin{document}
\title[Moser-Trudinger inequality revisited]{Aubin type almost sharp
Moser-Trudinger inequality revisited}
\author{Fengbo Hang}
\address{Courant Institute, New York University, 251 Mercer Street, New York
NY 10012}
\email{fengbo@cims.nyu.edu}

\begin{abstract}
We give a new proof of the almost sharp Moser-Trudinger inequality on
compact Riemannian manifolds based on the sharp Moser inequality on
Euclidean spaces. In particular we can lower the smoothness requirement of
the metric and apply the same approach to higher order Sobolev spaces and
manifolds with boundary under several boundary conditions.
\end{abstract}

\maketitle

\section{Introduction\label{sec1}}

Let $n\in \mathbb{N}$, $n\geq 2$. For $1<p<n$, the classical Sobolev
inequality states that there exists a positive constant $c\left( n,p\right) $
with%
\begin{equation}
\left\Vert \varphi \right\Vert _{L^{\frac{np}{n-p}}}\leq c\left( n,p\right)
\left\Vert \nabla \varphi \right\Vert _{L^{p}}  \label{eq1.1}
\end{equation}%
for any $\varphi \in C_{c}^{\infty }\left( \mathbb{R}^{n}\right) $. We
denote $S_{n,p}$ as the smallest possible choice for the constant $c\left(
n,p\right) $ i.e.%
\begin{equation}
S_{n,p}=\sup_{\varphi \in C_{c}^{\infty }\left( \mathbb{R}^{n}\right)
\backslash \left\{ 0\right\} }\frac{\left\Vert \varphi \right\Vert _{L^{%
\frac{np}{n-p}}}}{\left\Vert \nabla \varphi \right\Vert _{L^{p}}}.
\label{eq1.2}
\end{equation}%
Let $\left( M^{n},g\right) $ be a smooth compact Riemannian manifold of
dimension $n$ and $1<p<n$. Aubin \cite{Au1} showed that for any $\varepsilon
>0$, we have%
\begin{equation}
\left\Vert u\right\Vert _{L^{\frac{np}{n-p}}\left( M\right) }^{p}\leq \left(
S_{n,p}^{p}+\varepsilon \right) \left\Vert \nabla u\right\Vert _{L^{p}\left(
M\right) }^{p}+c\left( \varepsilon \right) \left\Vert u\right\Vert
_{L^{p}\left( M\right) }^{p}  \label{eq1.3}
\end{equation}%
for $u\in W^{1,p}\left( M\right) $. We call (\ref{eq1.3}) Aubin's almost
sharp Sobolev inequality. Aubin's almost sharp Sobolev inequality and its
closely related concentration compactness principle play fundamental role in
the study of semilinear equations (see \cite{He, Ln1,Ln2}).

On the other hand, let $B_{R}=B_{R}^{n}\subset \mathbb{R}^{n}$ be the open
ball with origin as center and $R$ as radius. In \cite{M}, it is shown that
for any $u\in W_{0}^{1,n}\left( B_{R}\right) \backslash \left\{ 0\right\} $,%
\begin{equation}
\int_{B_{R}}\exp \left( a_{n}\frac{\left\vert u\right\vert ^{\frac{n}{n-1}}}{%
\left\Vert \nabla u\right\Vert _{L^{n}\left( B_{R}\right) }^{\frac{n}{n-1}}}%
\right) dx\leq c\left( n\right) R^{n}.  \label{eq1.4}
\end{equation}%
Here%
\begin{equation}
a_{n}=n\left\vert \mathbb{S}^{n-1}\right\vert ^{\frac{1}{n-1}}.
\label{eq1.5}
\end{equation}%
$\left\vert \mathbb{S}^{n-1}\right\vert $ is the volume of $\mathbb{S}^{n-1}$
under the standard metric. Note that $a_{n}$ is sharp in the sense that if $%
a>0$ s.t.%
\begin{equation*}
\sup_{u\in W_{0}^{1,n}\left( B_{R}\right) \backslash \left\{ 0\right\}
}\int_{B_{R}}\exp \left( a\frac{\left\vert u\right\vert ^{\frac{n}{n-1}}}{%
\left\Vert \nabla u\right\Vert _{L^{n}\left( B_{R}\right) }^{\frac{n}{n-1}}}%
\right) dx<\infty \text{,}
\end{equation*}%
then we have $a\leq a_{n}$. For similar inequalities on general smooth
compact Riemannian manifolds, in \cite{F}, it is shown that if $\left(
M^{n},g\right) $ is a smooth compact Riemannian manifold, then for any $u\in
W^{1,n}\left( M\right) \backslash \left\{ 0\right\} $ with $\int_{M}ud\mu =0$
(here $\mu $ is the measure associated with $g$),%
\begin{equation}
\int_{M}\exp \left( a_{n}\frac{\left\vert u\right\vert ^{\frac{n}{n-1}}}{%
\left\Vert \nabla u\right\Vert _{L^{n}}^{\frac{n}{n-1}}}\right) d\mu \leq
c\left( M,g\right) .  \label{eq1.6}
\end{equation}%
The Moser-Trudinger inequality (\ref{eq1.6}) and its analogy are closely
related to spectral geometry through the classical Polyakov formula, Gauss
curvature equation and Q curvature equations (see for example \cite%
{Au2,CY1,CY2,M,OPS} and the references therein). A direct consequence of (%
\ref{eq1.6}) is for $u\in W^{1,n}\left( M\right) \backslash \left\{
0\right\} $ with $\int_{M}ud\mu =0$ and $\varepsilon >0$ small,%
\begin{equation}
\int_{M}\exp \left( \left( a_{n}-\varepsilon \right) \frac{\left\vert
u\right\vert ^{\frac{n}{n-1}}}{\left\Vert \nabla u\right\Vert _{L^{n}}^{%
\frac{n}{n-1}}}\right) d\mu \leq c\left( \varepsilon \right) <\infty .
\label{eq1.7}
\end{equation}%
We call (\ref{eq1.7}) as Aubin type almost sharp Moser-Trudinger inequality.
It is worth pointing out that for most applications, almost sharp
Moser-Trudinger inequality is sufficient (using suitable approximation
process when necessary, see for example Lemma \ref{lem3.2} and the
discussion after it). On the other hand, one can pass from (\ref{eq1.7}) to (%
\ref{eq1.6}) by blow-up analysis (see \cite{Li1, Li2}).

To continue we observe that Aubin's almost sharp Sobolev inequality follows
from (\ref{eq1.2}) by a smart but elementary application of decomposition of
unit (see \cite{Au1}). On the other hand, such passage from (\ref{eq1.4}) to
(\ref{eq1.7}) is missing. Fontana's proof for (\ref{eq1.6}) uses potential
theory approach by Adams \cite{A} and requires accurate estimate of the
Green's function. This method does not work well for functions satisfying
suitable boundary conditions (see \cite{Y2}). The main aim of this note is
to provide a direct and elementary method to deduce (\ref{eq1.7}) from (\ref%
{eq1.4}). Our approach is motivated from recent progress in concentration
compactness principle in critical dimension in \cite{CH,H}, where Aubin's
Moser-Trudinger-Onofri inequality on $\mathbb{S}^{n}$ in \cite{Au2} is
generalized to higher order moments cases. The sequence of inequalities are
motivated from similar inequalities on $\mathbb{S}^{1}$ (see \cite{GrS,OPS,W}%
). The key point in \cite{CH,H}, which is different from \cite{CCH,Ln1,Ln2},
is recognizing the importance of the value of defect measure at a single
point. This point of view will be crucial here as well (see Proposition \ref%
{prop1.1} below). An interesting gain of our approach is we only need the
metric to be continuous. Due to the elementary nature of our method, we can
easily adapt it to higher order Sobolev spaces and functions satisfying
various boundary conditions.

In the remaining part of this section, we will prove Theorem \ref{thm1.1},
which covers (\ref{eq1.7}). In Section \ref{sec2} we will adapt the approach
to functions on manifolds with nonempty boundary, either under Dirichlet
condition or no boundary condition at all. In Section \ref{sec3}, we prove
similar results for second order Sobolev spaces under various boundary
conditions and discuss its applications to Q curvature equations in
dimension 4. In Section \ref{sec4}, we briefly describe what can be done for
higher order Sobolev spaces.

The first fact is a qualitative property of Sobolev functions following from
(\ref{eq1.4}). It should be compared to \cite[Lemma 2.1]{CH} and \cite[Lemma
2.1]{H}.

\begin{lemma}
\label{lem1.1}Let $u\in W^{1,n}\left( B_{R}^{n}\right) $ and $a>0$, then%
\begin{equation}
\int_{B_{R}}e^{a\left\vert u\right\vert ^{\frac{n}{n-1}}}dx<\infty .
\label{eq1.8}
\end{equation}
\end{lemma}

\begin{proof}
We first treat the case $u\in W_{0}^{1,n}\left( B_{R}\right) $. If $u$ is
bounded the claim is clearly true. We assume $u$ is unbounded. For $%
\varepsilon >0$, a tiny number to be determined, we can find $v\in
C_{c}^{\infty }\left( B_{R}\right) $ such that%
\begin{equation*}
\left\Vert \nabla \left( u-v\right) \right\Vert _{L^{n}}<\varepsilon .
\end{equation*}%
Let $w=u-v$, then%
\begin{equation*}
\left\vert u\right\vert =\left\vert v+w\right\vert \leq \left\Vert
v\right\Vert _{L^{\infty }}+\left\vert w\right\vert .
\end{equation*}%
Hence%
\begin{equation*}
\left\vert u\right\vert ^{\frac{n}{n-1}}\leq 2^{\frac{1}{n-1}}\left\Vert
v\right\Vert _{L^{\infty }}^{\frac{n}{n-1}}+2^{\frac{1}{n-1}}\left\vert
w\right\vert ^{\frac{n}{n-1}}.
\end{equation*}%
It follows that%
\begin{equation*}
e^{a\left\vert u\right\vert ^{\frac{n}{n-1}}}\leq e^{2^{\frac{1}{n-1}%
}a\left\Vert v\right\Vert _{L^{\infty }}^{\frac{n}{n-1}}}e^{2^{\frac{1}{n-1}%
}a\left\vert w\right\vert ^{\frac{n}{n-1}}}\leq e^{2^{\frac{1}{n-1}%
}a\left\Vert v\right\Vert _{L^{\infty }}^{\frac{n}{n-1}}}e^{a_{n}\frac{%
\left\vert w\right\vert ^{\frac{n}{n-1}}}{\left\Vert \nabla w\right\Vert
_{L^{n}}^{\frac{n}{n-1}}}}
\end{equation*}%
if $\varepsilon $ is small enough. Using (\ref{eq1.4}) we see%
\begin{equation*}
\int_{B_{R}}e^{a\left\vert u\right\vert ^{\frac{n}{n-1}}}dx\leq c\left(
n\right) e^{2^{\frac{1}{n-1}}a\left\Vert v\right\Vert _{L^{\infty }}^{\frac{n%
}{n-1}}}R^{n}<\infty .
\end{equation*}

In general, if $u\in W^{1,n}\left( B_{R}\right) $, then we can find $%
\widetilde{u}\in W_{0}^{1,n}\left( B_{2R}\right) $ such that $\left. 
\widetilde{u}\right\vert _{B_{R}}=u$. Hence%
\begin{equation*}
\int_{B_{R}}e^{a\left\vert u\right\vert ^{\frac{n}{n-1}}}dx\leq
\int_{B_{2R}}e^{a\left\vert \widetilde{u}\right\vert ^{\frac{n}{n-1}%
}}dx<\infty
\end{equation*}%
by the previous discussion.
\end{proof}

Next we localize \cite[Lemma 2.2]{CH} and \cite[Proposition 2.1]{H}.

\begin{proposition}
\label{prop1.1}Let $0<R\leq 1$, $g$ be a continuous Riemannian metric on $%
\overline{B_{R}^{n}}$. Assume $u_{i}\in W^{1,n}\left( B_{R}\right) $, $%
u_{i}\rightharpoonup u$ weakly in $W^{1,n}\left( B_{R}\right) $ and%
\begin{equation*}
\left\vert \nabla _{g}u_{i}\right\vert _{g}^{n}d\mu \rightarrow \left\vert
\nabla _{g}u\right\vert _{g}^{n}d\mu +\sigma \text{ as measure on }B_{R}%
\text{.}
\end{equation*}%
If $0<p<\sigma \left( \left\{ 0\right\} \right) ^{-\frac{1}{n-1}}$, then for
some $r>0$,%
\begin{equation}
\sup_{i}\int_{B_{r}}e^{a_{n}p\left\vert u_{i}\right\vert ^{\frac{n}{n-1}%
}}d\mu <\infty .  \label{eq1.9}
\end{equation}%
Here%
\begin{equation}
a_{n}=n\left\vert \mathbb{S}^{n-1}\right\vert ^{\frac{1}{n-1}}.
\label{eq1.10}
\end{equation}
\end{proposition}

\begin{remark}
\label{rmk1.1}\cite[Lemma 2.2]{CH} and \cite[Proposition 2.1]{H} follows
from Proposition \ref{prop1.1}. Hence to derive the main results in \cite%
{CH,H}, we may use Moser's sharp inequality on Euclidean spaces instead of
Fontana's Moser-Trudinger inequality (\ref{eq1.6}).
\end{remark}

To prove Proposition \ref{prop1.1}, we first observe that by a linear
changing of variable and shrinking $R$ if necessary we can assume $%
g=g_{ij}dx_{i}dx_{j}$ with $g_{ij}\left( 0\right) =\delta _{ij}$. Fix $%
p_{1}\in \left( p,\sigma \left( \left\{ 0\right\} \right) ^{-\frac{1}{n-1}%
}\right) $, then%
\begin{equation}
\sigma \left( \left\{ 0\right\} \right) <\frac{1}{p_{1}^{n-1}}.
\label{eq1.11}
\end{equation}%
We can find $\varepsilon >0$ such that%
\begin{equation}
\left( 1+\varepsilon \right) \sigma \left( \left\{ 0\right\} \right) <\frac{1%
}{p_{1}^{n-1}}  \label{eq1.12}
\end{equation}%
and%
\begin{equation}
\left( 1+\varepsilon \right) p<p_{1}.  \label{eq1.13}
\end{equation}%
Let $v_{i}=u_{i}-u$, then $v_{i}\rightharpoonup 0$ weakly in $W^{1,n}\left(
B_{R}\right) $, $v_{i}\rightarrow 0$ in $L^{n}\left( B_{R}\right) $. To
continue, we note that on $B_{R}$ we have two Riemannian metrics: The
standard Euclidean metric and $g$. For a function $f$ on $B_{R}$, we write%
\begin{equation}
\left\Vert f\right\Vert _{L^{n}}=\left\Vert f\right\Vert _{L^{n}\left(
B_{R}\right) }=\left\Vert f\right\Vert _{L^{n}\left( B_{R},dx\right) },
\label{eq1.14}
\end{equation}%
here $dx$ is the standard Lebesgue measure. This should be compared with $%
\left\Vert f\right\Vert _{L^{n}\left( B_{R},d\mu \right) }$. We write $%
\nabla f$ as the usual gradient of $f$ i.e. the gradient with respect to
Euclidean metric and $\nabla _{g}f$ as the gradient of $f$ with respect to
metric $g$.

Let $0<R_{1}<R$ be a small number to be determined. For any $\varphi \in
C_{c}^{\infty }\left( B_{R_{1}}\right) $, we have%
\begin{eqnarray*}
&&\left\Vert \nabla \left( \varphi v_{i}\right) \right\Vert _{L^{n}\left(
B_{R}\right) }^{n} \\
&\leq &\left( \left\Vert \varphi \nabla v_{i}\right\Vert _{L^{n}}+\left\Vert
v_{i}\nabla \varphi \right\Vert _{L^{n}}\right) ^{n} \\
&\leq &\left( \left\Vert \varphi \nabla u_{i}\right\Vert _{L^{n}}+\left\Vert
\varphi \nabla u\right\Vert _{L^{n}}+\left\Vert v_{i}\nabla \varphi
\right\Vert _{L^{n}}\right) ^{n} \\
&\leq &\left( 1+\frac{\varepsilon }{2}\right) \left\Vert \varphi \nabla
u_{i}\right\Vert _{L^{n}}^{n}+c\left( \varepsilon \right) \left\Vert \varphi
\nabla u\right\Vert _{L^{n}}^{n}+c\left( \varepsilon \right) \left\Vert
v_{i}\nabla \varphi \right\Vert _{L^{n}}^{n} \\
&\leq &\left( 1+\varepsilon \right) \left\Vert \varphi \nabla
_{g}u_{i}\right\Vert _{L^{n}\left( B_{R},d\mu \right) }^{n}+c\left(
\varepsilon \right) \left\Vert \varphi \nabla u\right\Vert
_{L^{n}}^{n}+c\left( \varepsilon \right) \left\Vert v_{i}\nabla \varphi
\right\Vert _{L^{n}}^{n}.
\end{eqnarray*}%
Here we have used the continuity of $g_{ij}$, $g_{ij}\left( 0\right) =\delta
_{ij}$ and the smallness of $R_{1}$. Hence%
\begin{eqnarray*}
&&\lim \sup_{i\rightarrow \infty }\left\Vert \nabla \left( \varphi
v_{i}\right) \right\Vert _{L^{n}\left( B_{R}\right) }^{n} \\
&\leq &\left( 1+\varepsilon \right) \int_{B_{R}}\left\vert \varphi
\right\vert ^{n}d\sigma +\left( 1+\varepsilon \right) \int_{B_{R}}\left\vert
\varphi \right\vert ^{n}\left\vert \nabla _{g}u\right\vert _{g}^{n}d\mu
+c\left( \varepsilon \right) \int_{B_{R}}\left\vert \varphi \right\vert
^{n}\left\vert \nabla u\right\vert ^{n}dx \\
&\leq &\left( 1+\varepsilon \right) \int_{B_{R}}\left\vert \varphi
\right\vert ^{n}d\sigma +c\left( \varepsilon \right) \int_{B_{R}}\left\vert
\varphi \right\vert ^{n}\left\vert \nabla u\right\vert ^{n}dx.
\end{eqnarray*}%
Since $\left( 1+\varepsilon \right) \sigma \left( \left\{ 0\right\} \right) <%
\frac{1}{p_{1}^{n-1}}$, we can choose $\varphi \in C_{c}^{\infty }\left(
B_{R_{1}}\right) $ such that $\left. \varphi \right\vert _{B_{r}}=1$ for
some $r>0$ and%
\begin{equation}
\left( 1+\varepsilon \right) \int_{B_{R}}\left\vert \varphi \right\vert
^{n}d\sigma +c\left( \varepsilon \right) \int_{B_{R}}\left\vert \varphi
\right\vert ^{n}\left\vert \nabla u\right\vert ^{n}dx<\frac{1}{p_{1}^{n-1}}.
\label{eq1.15}
\end{equation}%
Hence for $i$ large enough, we have%
\begin{equation}
\left\Vert \nabla \left( \varphi v_{i}\right) \right\Vert _{L^{n}\left(
B_{R}\right) }^{n}<\frac{1}{p_{1}^{n-1}}.  \label{eq1.16}
\end{equation}%
This implies%
\begin{equation}
\left\Vert \nabla \left( \varphi v_{i}\right) \right\Vert _{L^{n}\left(
B_{R}\right) }^{\frac{n}{n-1}}<\frac{1}{p_{1}}.  \label{eq1.17}
\end{equation}%
On the other hand,%
\begin{equation*}
\left\vert u_{i}\right\vert ^{\frac{n}{n-1}}\leq \left( \left\vert
v_{i}\right\vert +\left\vert u\right\vert \right) ^{\frac{n}{n-1}}\leq
\left( 1+\varepsilon \right) \left\vert v_{i}\right\vert ^{\frac{n}{n-1}%
}+c\left( \varepsilon \right) \left\vert u\right\vert ^{\frac{n}{n-1}},
\end{equation*}%
hence%
\begin{equation*}
e^{a_{n}\left\vert u_{i}\right\vert ^{\frac{n}{n-1}}}\leq e^{\left(
1+\varepsilon \right) a_{n}\left\vert v_{i}\right\vert ^{\frac{n}{n-1}%
}}e^{c\left( \varepsilon \right) \left\vert u\right\vert ^{\frac{n}{n-1}}}.
\end{equation*}%
It follows from (\ref{eq1.4}) that%
\begin{equation*}
\int_{B_{r}}e^{a_{n}p_{1}\left\vert v_{i}\right\vert ^{\frac{n}{n-1}}}d\mu
\leq c\int_{B_{R}}e^{a_{n}p_{1}\left\vert \varphi v_{i}\right\vert ^{\frac{n%
}{n-1}}}dx\leq c\int_{B_{R}}e^{a_{n}\frac{\left\vert \varphi
v_{i}\right\vert ^{\frac{n}{n-1}}}{\left\Vert \nabla \left( \varphi
v_{i}\right) \right\Vert _{L^{n}}^{\frac{n}{n-1}}}}dx\leq c.
\end{equation*}%
Using $p<\frac{p_{1}}{1+\varepsilon }$ and Lemma \ref{lem1.1}, it follows
from Holder's inequality that $e^{a_{n}\left\vert u_{i}\right\vert ^{\frac{n%
}{n-1}}}$ is bounded in $L^{p}\left( B_{r},d\mu \right) $. This finishes the
proof of Proposition \ref{prop1.1}.

\begin{theorem}
\label{thm1.1}Let $\left( M^{n},g\right) $ be a $C^{1}$ compact manifold
with a continuous Riemannian metric $g$. If $E\subset M$ with $\mu \left(
E\right) \geq \delta >0$, $u\in W^{1,n}\left( M\right) \backslash \left\{
0\right\} $ with $u_{E}=0$ (here $u_{E}=\frac{1}{\mu \left( E\right) }%
\int_{E}ud\mu $), then for any $0<a<a_{n}=n\left\vert \mathbb{S}%
^{n-1}\right\vert ^{\frac{1}{n-1}}$,%
\begin{equation}
\int_{M}e^{a\frac{\left\vert u\right\vert ^{\frac{n}{n-1}}}{\left\Vert
\nabla u\right\Vert _{L^{n}}^{\frac{n}{n-1}}}}d\mu \leq c\left( a,\delta
\right) <\infty .  \label{eq1.18}
\end{equation}
\end{theorem}

\begin{remark}
\label{rmk1.2}It is interesting that we can find a continuous Riemannian
metric on $M$ such that%
\begin{equation}
\sup_{\substack{ u\in W^{1,n}\left( M\right) \backslash \left\{ 0\right\} 
\\ u_{M}=0}}\int_{M}\exp \left( a_{n}\frac{\left\vert u\right\vert ^{\frac{n%
}{n-1}}}{\left\Vert \nabla u\right\Vert _{L^{n}\left( M\right) }^{\frac{n}{%
n-1}}}\right) d\mu =\infty .  \label{eq1.19}
\end{equation}%
Indeed let $\left( U,\phi \right) $ be a coordinate on $M$ such that $\phi
\left( U\right) =B_{2}$. For convenience we identify $U$ with $B_{2}$. Let%
\begin{equation}
\alpha \left( r\right) =\frac{1}{\sqrt{\log \frac{1}{r}+4}}\quad \text{for }%
0<r<1\text{,}  \label{eq1.20}
\end{equation}%
then $0\leq \alpha \leq \frac{1}{2}$ and%
\begin{equation}
\int_{0}^{1}\frac{\alpha \left( r\right) }{r}dr=\infty .  \label{eq1.21}
\end{equation}%
We can find a continuous Riemannian metric on $M$ such that on $B_{1}$, $%
g=g_{ij}dx_{i}dx_{j}$ with%
\begin{equation}
g_{ij}=\left( 1-\alpha \left( \left\vert x\right\vert \right) \right) ^{-%
\frac{2}{n}}\frac{x_{i}}{\left\vert x\right\vert }\frac{x_{j}}{\left\vert
x\right\vert }+\left( 1-\alpha \left( \left\vert x\right\vert \right)
\right) ^{\frac{2}{n\left( n-1\right) }}\left( \delta _{ij}-\frac{x_{i}}{%
\left\vert x\right\vert }\frac{x_{j}}{\left\vert x\right\vert }\right) .
\label{eq1.22}
\end{equation}%
Hence%
\begin{equation}
g^{ij}=\left( 1-\alpha \left( \left\vert x\right\vert \right) \right) ^{%
\frac{2}{n}}\frac{x_{i}}{\left\vert x\right\vert }\frac{x_{j}}{\left\vert
x\right\vert }+\left( 1-\alpha \left( \left\vert x\right\vert \right)
\right) ^{-\frac{2}{n\left( n-1\right) }}\left( \delta _{ij}-\frac{x_{i}}{%
\left\vert x\right\vert }\frac{x_{j}}{\left\vert x\right\vert }\right) .
\label{eq1.23}
\end{equation}%
For $0<\varepsilon <1$, we define%
\begin{equation*}
v_{\varepsilon }\left( x\right) =\left\{ 
\begin{tabular}{ll}
$\log \varepsilon ,$ & if $\left\vert x\right\vert \leq \varepsilon $; \\ 
$\log \left\vert x\right\vert ,$ & if $\varepsilon <\left\vert x\right\vert
<1$; \\ 
$0,$ & if $x\in M\backslash B_{1}$.%
\end{tabular}%
\right.
\end{equation*}%
Then%
\begin{eqnarray*}
\left\Vert \nabla v_{\varepsilon }\right\Vert _{L^{n}}^{n}
&=&\int_{\varepsilon <\left\vert x\right\vert <1}\left( g^{ij}\frac{x_{i}}{%
\left\vert x\right\vert ^{2}}\frac{x_{j}}{\left\vert x\right\vert ^{2}}%
\right) ^{\frac{n}{2}}dx \\
&=&\int_{\varepsilon <\left\vert x\right\vert <1}\frac{1-\alpha \left(
\left\vert x\right\vert \right) }{\left\vert x\right\vert ^{n}}dx \\
&=&\left\vert \mathbb{S}^{n-1}\right\vert \left( \log \frac{1}{\varepsilon }%
-\int_{\varepsilon }^{1}\frac{\alpha \left( r\right) }{r}dr\right) .
\end{eqnarray*}%
On the other hand, $v_{\varepsilon ,M}=O\left( 1\right) $ as $\varepsilon
\rightarrow 0^{+}$. Letting $u_{\varepsilon }=v_{\varepsilon
}-v_{\varepsilon ,M}$, we claim%
\begin{equation}
\lim_{\varepsilon \rightarrow 0^{+}}\int_{M}\exp \left( a_{n}\frac{%
\left\vert u_{\varepsilon }\right\vert ^{\frac{n}{n-1}}}{\left\Vert \nabla
u_{\varepsilon }\right\Vert _{L^{n}\left( M\right) }^{\frac{n}{n-1}}}\right)
d\mu =\infty .  \label{eq1.24}
\end{equation}%
Indeed,%
\begin{eqnarray*}
&&\int_{M}\exp \left( a_{n}\frac{\left\vert u_{\varepsilon }\right\vert ^{%
\frac{n}{n-1}}}{\left\Vert \nabla u_{\varepsilon }\right\Vert _{L^{n}\left(
M\right) }^{\frac{n}{n-1}}}\right) d\mu \\
&\geq &\int_{B_{\varepsilon }}\exp \left( a_{n}\frac{\left\vert \log
\varepsilon -v_{\varepsilon ,M}\right\vert ^{\frac{n}{n-1}}}{\left\Vert
\nabla v_{\varepsilon }\right\Vert _{L^{n}\left( M\right) }^{\frac{n}{n-1}}}%
\right) dx \\
&=&\left\vert B_{1}\right\vert \varepsilon ^{n}\exp \left( n\frac{\left\vert
\log \frac{1}{\varepsilon }+v_{\varepsilon ,M}\right\vert ^{\frac{n}{n-1}}}{%
\left( \log \frac{1}{\varepsilon }-\int_{\varepsilon }^{1}\frac{\alpha
\left( r\right) }{r}dr\right) ^{\frac{1}{n-1}}}\right) \\
&\geq &c\left( n\right) \exp \left[ n\log \frac{1}{\varepsilon }\cdot \left(
1-\frac{c_{1}}{\log \frac{1}{\varepsilon }}\right) \left( 1+\frac{1}{n-1}%
\frac{\int_{\varepsilon }^{1}\frac{\alpha \left( r\right) }{r}dr}{\log \frac{%
1}{\varepsilon }}\right) -n\log \frac{1}{\varepsilon }\right] \\
&\geq &c\left( n\right) \exp \left( \frac{n}{n-1}\int_{\varepsilon }^{1}%
\frac{\alpha \left( r\right) }{r}dr-c\right) \\
&\geq &c\exp \left( \frac{n}{n-1}\int_{\varepsilon }^{1}\frac{\alpha \left(
r\right) }{r}dr\right) .
\end{eqnarray*}%
This estimate together with (\ref{eq1.21}) implies (\ref{eq1.24}). (\ref%
{eq1.19}) follows.
\end{remark}

\begin{proof}[Proof of Theorem \protect\ref{thm1.1}]
For any $v\in W^{1,n}\left( M\right) $, we know%
\begin{equation*}
\left\Vert v-v_{M}\right\Vert _{L^{n}}\leq c\left\Vert \nabla v\right\Vert
_{L^{n}}.
\end{equation*}%
On the other hand,%
\begin{equation*}
\left\Vert v-v_{E}\right\Vert _{L^{n}}\leq \frac{c}{\mu \left( E\right) ^{%
\frac{1}{n}}}\left\Vert v\right\Vert _{L^{n}}\leq c\left( \delta \right)
\left\Vert v\right\Vert _{L^{n}}.
\end{equation*}%
Replacing $v$ by $v-v_{M}$, we see%
\begin{equation*}
\left\Vert v-v_{E}\right\Vert _{L^{n}}\leq c\left( \delta \right) \left\Vert
v-v_{M}\right\Vert _{L^{n}}\leq c\left( \delta \right) \left\Vert \nabla
v\right\Vert _{L^{n}}.
\end{equation*}

If (\ref{eq1.18}) is not true, then we can find a sequence $u_{i}\in
W^{1,n}\left( M\right) $, $E_{i}\subset M$ with $\mu \left( E_{i}\right)
\geq \delta $, $u_{i,E_{i}}=0$, $\left\Vert \nabla u_{i}\right\Vert
_{L^{n}}=1$ and%
\begin{equation*}
\int_{M}e^{a\left\vert u_{i}\right\vert ^{\frac{n}{n-1}}}d\mu \rightarrow
\infty
\end{equation*}%
as $i\rightarrow \infty $. Since%
\begin{equation*}
\left\Vert u_{i}\right\Vert _{L^{n}}=\left\Vert u_{i}-u_{i,E_{i}}\right\Vert
_{L^{n}}\leq c\left( \delta \right) \left\Vert \nabla u_{i}\right\Vert
_{L^{n}}=c\left( \delta \right) ,
\end{equation*}%
we see $u_{i}$ is bounded in $W^{1,n}\left( M\right) $. Hence after passing
to a subsequence we can find $u\in W^{1,n}\left( M\right) $ such that $%
u_{i}\rightharpoonup u$ weakly in $W^{1,n}\left( M\right) $ and a measure on 
$M$, $\sigma $ such that%
\begin{equation*}
\left\vert \nabla u_{i}\right\vert ^{n}d\mu \rightarrow \left\vert \nabla
u\right\vert ^{n}d\mu +\sigma
\end{equation*}%
as measure. Note that $\sigma \left( M\right) \leq 1$. For any $x\in M$,
since%
\begin{equation*}
0<\frac{a}{a_{n}}<\sigma \left( \left\{ x\right\} \right) ^{-\frac{1}{n-1}},
\end{equation*}%
it follows from Proposition \ref{prop1.1} that for some $r>0$, we have%
\begin{equation*}
\sup_{i}\int_{B_{r}\left( x\right) }e^{a\left\vert u_{i}\right\vert ^{\frac{n%
}{n-1}}}d\mu <\infty .
\end{equation*}%
A covering argument implies%
\begin{equation*}
\sup_{i}\int_{M}e^{a\left\vert u_{i}\right\vert ^{\frac{n}{n-1}}}d\mu
<\infty .
\end{equation*}%
This contradicts with the choice of $u_{i}$.
\end{proof}

The argument above is flexible and can be applied to other cases as well.

\begin{example}
\label{ex1.1}Let $M^{n}$ be a $C^{1}$ compact manifold with a $C^{0}$
Riemannian metric $g$. Denote%
\begin{equation}
\kappa \left( M,g\right) =\inf_{\substack{ u\in W^{1,n}\left( M\right)
\backslash \left\{ 0\right\}  \\ u_{M}=0}}\frac{\left\Vert \nabla
u\right\Vert _{L^{n}}}{\left\Vert u\right\Vert _{L^{n}}}.  \label{eq1.25}
\end{equation}%
It follows from Poincare inequality that $\kappa \left( M,g\right) $ is a
positive number. Assume $0\leq \kappa <\kappa \left( M,g\right) $, $%
0<a<a_{n} $, $u\in W^{1,n}\left( M\right) $ with $u_{M}=0$ and%
\begin{equation}
\left\Vert \nabla u\right\Vert _{L^{n}}^{n}-\kappa ^{n}\left\Vert
u\right\Vert _{L^{n}}^{n}\leq 1,  \label{eq1.26}
\end{equation}%
then we have%
\begin{equation}
\int_{M}e^{a\left\vert u\right\vert ^{\frac{n}{n-1}}}d\mu \leq c\left(
\kappa ,a\right) <\infty .  \label{eq1.27}
\end{equation}
\end{example}

\begin{proof}
If the claim is not true, then we can find a sequence $u_{i}\in
W^{1,n}\left( M\right) $ s.t. $u_{i,M}=0$, $\left\Vert \nabla
u_{i}\right\Vert _{L^{n}}^{n}-\kappa ^{n}\left\Vert u_{i}\right\Vert
_{L^{n}}^{n}\leq 1$ and $\int_{M}e^{a\left\vert u_{i}\right\vert ^{\frac{n}{%
n-1}}}d\mu \rightarrow \infty $ as $i\rightarrow \infty $. By the definition
of $\kappa \left( M,g\right) $, we see%
\begin{equation*}
\left( \kappa \left( M,g\right) ^{n}-\kappa ^{n}\right) \left\Vert
u_{i}\right\Vert _{L^{n}}^{n}\leq \left\Vert \nabla u_{i}\right\Vert
_{L^{n}}^{n}-\kappa ^{n}\left\Vert u_{i}\right\Vert _{L^{n}}^{n}\leq 1,
\end{equation*}%
hence%
\begin{equation*}
\left\Vert u_{i}\right\Vert _{L^{n}}^{n}\leq \frac{1}{\kappa \left(
M,g\right) ^{n}-\kappa ^{n}}.
\end{equation*}%
This together with $\left\Vert \nabla u_{i}\right\Vert _{L^{n}}^{n}-\kappa
^{n}\left\Vert u_{i}\right\Vert _{L^{n}}^{n}\leq 1$ implies%
\begin{equation*}
\left\Vert \nabla u_{i}\right\Vert _{L^{n}}^{n}\leq \frac{\kappa \left(
M,g\right) ^{n}}{\kappa \left( M,g\right) ^{n}-\kappa ^{n}}.
\end{equation*}%
In other words, $u_{i}$ is bounded in $W^{1,n}\left( M\right) $. After
passing to a subsequence, we can find $u\in W^{1,n}\left( M\right) $ s.t. $%
u_{i}\rightharpoonup u$ weakly in $W^{1,n}\left( M\right) $, $%
u_{i}\rightarrow u$ in $L^{n}\left( M\right) $ and%
\begin{equation*}
\left\vert \nabla u_{i}\right\vert ^{n}d\mu \rightarrow \left\vert \nabla
u\right\vert ^{n}d\mu +\sigma
\end{equation*}%
as measure. In particular, $u_{M}=0$. On the other hand,%
\begin{equation*}
\left\Vert \nabla u\right\Vert _{L^{n}}^{n}+\sigma \left( M\right) -\kappa
^{n}\left\Vert u\right\Vert _{L^{n}}^{n}\leq 1.
\end{equation*}%
Since $\left\Vert \nabla u\right\Vert _{L^{n}}^{n}-\kappa ^{n}\left\Vert
u\right\Vert _{L^{n}}^{n}\geq 0$, we see $\sigma \left( M\right) \leq 1$.
For any $x\in M$, since%
\begin{equation*}
0<\frac{a}{a_{n}}<\sigma \left( \left\{ x\right\} \right) ^{-\frac{1}{n-1}},
\end{equation*}%
it follows from Proposition \ref{prop1.1} that for some $r>0$, we have%
\begin{equation*}
\sup_{i}\int_{B_{r}\left( x\right) }e^{a\left\vert u_{i}\right\vert ^{\frac{n%
}{n-1}}}d\mu <\infty .
\end{equation*}%
This together with compactness of $M$ implies%
\begin{equation*}
\sup_{i}\int_{M}e^{a\left\vert u_{i}\right\vert ^{\frac{n}{n-1}}}d\mu
<\infty .
\end{equation*}%
It gives us a contradiction by the choice of $u_{i}$ at the beginning.
\end{proof}

\section{Functions on manifolds with nonempty boundary\label{sec2}}

In this section, we will study functions on manifolds with boundary under no
boundary conditions or Dirichlet boundary condition. At first we need to
discuss functions on half ball. Let $R>0$, we denote%
\begin{equation}
B_{R}^{+}=\left\{ x\in B_{R}:x=\left( x^{\prime },x_{n}\right) ,x^{\prime
}\in \mathbb{R}^{n-1},x_{n}\geq 0\right\} .  \label{eq2.1}
\end{equation}

\begin{lemma}
\label{lem2.1}Assume $u\in W^{1,n}\left( B_{R}^{+}\right) \backslash \left\{
0\right\} $ such that $u\left( x\right) =0$ for $\left\vert x\right\vert $
close to $R$, then%
\begin{equation}
\int_{B_{R}^{+}}e^{2^{-\frac{1}{n-1}}a_{n}\frac{\left\vert u\right\vert ^{%
\frac{n}{n-1}}}{\left\Vert \nabla u\right\Vert _{L^{n}\left(
B_{R}^{+}\right) }^{\frac{n}{n-1}}}}dx\leq c\left( n\right) R^{n}.
\label{eq2.2}
\end{equation}
\end{lemma}

\begin{proof}
For $\left\vert x\right\vert <R$, let%
\begin{equation*}
v\left( x\right) =\left\{ 
\begin{array}{cc}
u\left( x\right) , & \text{if }x_{n}>0; \\ 
u\left( x^{\prime },-x_{n}\right) , & \text{if }x_{n}<0\text{.}%
\end{array}%
\right.
\end{equation*}%
Then $v\in W_{0}^{1,n}\left( B_{R}\right) $ with $\left\Vert \nabla
v\right\Vert _{L^{n}\left( B_{R}\right) }^{n}=2\left\Vert \nabla
u\right\Vert _{L^{n}\left( B_{R}^{+}\right) }^{n}$. Using (\ref{eq1.4}) we
see%
\begin{eqnarray*}
\int_{B_{R}^{+}}e^{2^{-\frac{1}{n-1}}a_{n}\frac{\left\vert u\right\vert ^{%
\frac{n}{n-1}}}{\left\Vert \nabla u\right\Vert _{L^{n}\left(
B_{R}^{+}\right) }^{\frac{n}{n-1}}}}dx &=&\int_{B_{R}^{+}}e^{a_{n}\frac{%
\left\vert u\right\vert ^{\frac{n}{n-1}}}{\left\Vert \nabla v\right\Vert
_{L^{n}\left( B_{R}\right) }^{\frac{n}{n-1}}}}dx \\
&\leq &\int_{B_{R}}e^{a_{n}\frac{\left\vert v\right\vert ^{\frac{n}{n-1}}}{%
\left\Vert \nabla v\right\Vert _{L^{n}\left( B_{R}\right) }^{\frac{n}{n-1}}}%
}dx \\
&\leq &c\left( n\right) R^{n}.
\end{eqnarray*}
\end{proof}

We also need a qualitative property of Sobolev functions.

\begin{lemma}
\label{lem2.2}Let $u\in W^{1,n}\left( B_{R}^{+}\right) $ and $a>0$, then%
\begin{equation}
\int_{B_{R}^{+}}e^{a\left\vert u\right\vert ^{\frac{n}{n-1}}}dx<\infty .
\label{eq2.3}
\end{equation}
\end{lemma}

\begin{proof}
We can find $\widetilde{u}\in W^{1,n}\left( B_{R}\right) $ such that $\left. 
\widetilde{u}\right\vert _{B_{R}^{+}}=u$. Then it follows from Lemma \ref%
{lem1.1} that%
\begin{equation*}
\int_{B_{R}^{+}}e^{a\left\vert u\right\vert ^{\frac{n}{n-1}}}dx\leq
\int_{B_{R}}e^{a\left\vert \widetilde{u}\right\vert ^{\frac{n}{n-1}%
}}dx<\infty .
\end{equation*}
\end{proof}

\begin{proposition}
\label{prop2.1}Let $0<R\leq 1$, $g$ be a continuous Riemannian metric on $%
\overline{B_{R}^{+}}$. Assume $u_{i}\in W^{1,n}\left( B_{R}^{+}\right) $, $%
u_{i}\rightharpoonup u$ weakly in $W^{1,n}\left( B_{R}^{+}\right) $ and%
\begin{equation*}
\left\vert \nabla _{g}u_{i}\right\vert _{g}^{n}d\mu \rightarrow \left\vert
\nabla _{g}u\right\vert _{g}^{n}d\mu +\sigma \text{ as measure on }B_{R}^{+}%
\text{.}
\end{equation*}%
If $0<p<\sigma \left( \left\{ 0\right\} \right) ^{-\frac{1}{n-1}}$, then
there exists $r>0$ such that%
\begin{equation}
\sup_{i}\int_{B_{r}^{+}}e^{2^{-\frac{1}{n-1}}a_{n}p\left\vert
u_{i}\right\vert ^{\frac{n}{n-1}}}d\mu <\infty .  \label{eq2.4}
\end{equation}%
Here%
\begin{equation}
a_{n}=n\left\vert \mathbb{S}^{n-1}\right\vert ^{\frac{1}{n-1}}.
\label{eq2.5}
\end{equation}
\end{proposition}

\begin{proof}
By linear changing of variable and shrinking $R$ if necessary we can assume $%
g=g_{ij}dx_{i}dx_{j}$ with $g_{ij}\left( 0\right) =\delta _{ij}$. Fix $%
p_{1}\in \left( p,\sigma \left( \left\{ 0\right\} \right) ^{-\frac{1}{n-1}%
}\right) $, then%
\begin{equation}
\sigma \left( \left\{ 0\right\} \right) <\frac{1}{p_{1}^{n-1}}.
\label{eq2.6}
\end{equation}%
We can find $\varepsilon >0$ such that%
\begin{equation}
\left( 1+\varepsilon \right) \sigma \left( \left\{ 0\right\} \right) <\frac{1%
}{p_{1}^{n-1}}  \label{eq2.7}
\end{equation}%
and%
\begin{equation}
\left( 1+\varepsilon \right) p<p_{1}.  \label{eq2.8}
\end{equation}%
Let $v_{i}=u_{i}-u$, then $v_{i}\rightharpoonup 0$ weakly in $W^{1,n}\left(
B_{R}^{+}\right) $, $v_{i}\rightarrow 0$ in $L^{n}\left( B_{R}^{+}\right) $.
Let $0<R_{1}<R$ be a small number to be determined. For any $\varphi \in
C_{c}^{\infty }\left( B_{R_{1}}^{+}\right) $, we have%
\begin{eqnarray*}
&&\left\Vert \nabla \left( \varphi v_{i}\right) \right\Vert _{L^{n}\left(
B_{R}^{+}\right) }^{n} \\
&\leq &\left( \left\Vert \varphi \nabla v_{i}\right\Vert _{L^{n}}+\left\Vert
v_{i}\nabla \varphi \right\Vert _{L^{n}}\right) ^{n} \\
&\leq &\left( \left\Vert \varphi \nabla u_{i}\right\Vert _{L^{n}}+\left\Vert
\varphi \nabla u\right\Vert _{L^{n}}+\left\Vert v_{i}\nabla \varphi
\right\Vert _{L^{n}}\right) ^{n} \\
&\leq &\left( 1+\frac{\varepsilon }{2}\right) \left\Vert \varphi \nabla
u_{i}\right\Vert _{L^{n}}^{n}+c\left( \varepsilon \right) \left\Vert \varphi
\nabla u\right\Vert _{L^{n}}^{n}+c\left( \varepsilon \right) \left\Vert
v_{i}\nabla \varphi \right\Vert _{L^{n}}^{n} \\
&\leq &\left( 1+\varepsilon \right) \left\Vert \varphi \nabla
_{g}u_{i}\right\Vert _{L^{n}\left( B_{R}^{+},d\mu \right) }^{n}+c\left(
\varepsilon \right) \left\Vert \varphi \nabla u\right\Vert
_{L^{n}}^{n}+c\left( \varepsilon \right) \left\Vert v_{i}\nabla \varphi
\right\Vert _{L^{n}}^{n}.
\end{eqnarray*}%
Hence%
\begin{eqnarray*}
&&\lim \sup_{i\rightarrow \infty }\left\Vert \nabla \left( \varphi
v_{i}\right) \right\Vert _{L^{n}}^{n} \\
&\leq &\left( 1+\varepsilon \right) \int_{B_{R}^{+}}\left\vert \varphi
\right\vert ^{n}d\sigma +\left( 1+\varepsilon \right)
\int_{B_{R}^{+}}\left\vert \varphi \right\vert ^{n}\left\vert \nabla
_{g}u\right\vert _{g}^{n}d\mu +c\left( \varepsilon \right)
\int_{B_{R}^{+}}\left\vert \varphi \right\vert ^{n}\left\vert \nabla
u\right\vert ^{n}dx \\
&\leq &\left( 1+\varepsilon \right) \int_{B_{R}^{+}}\left\vert \varphi
\right\vert ^{n}d\sigma +c\left( \varepsilon \right)
\int_{B_{R}^{+}}\left\vert \varphi \right\vert ^{n}\left\vert \nabla
u\right\vert ^{n}dx.
\end{eqnarray*}%
Since $\left( 1+\varepsilon \right) \sigma \left( \left\{ 0\right\} \right) <%
\frac{1}{p_{1}^{n-1}}$, we can choose $\varphi \in C_{c}^{\infty }\left(
B_{R_{1}}^{+}\right) $ such that $\left. \varphi \right\vert _{B_{r}^{+}}=1$
for some $r>0$ and%
\begin{equation}
\left( 1+\varepsilon \right) \int_{B_{R}^{+}}\left\vert \varphi \right\vert
^{n}d\sigma +c\left( \varepsilon \right) \int_{B_{R}^{+}}\left\vert \varphi
\right\vert ^{n}\left\vert \nabla u\right\vert ^{n}dx<\frac{1}{p_{1}^{n-1}}.
\label{eq2.9}
\end{equation}%
Hence for $i$ large enough, we have%
\begin{equation}
\left\Vert \nabla \left( \varphi v_{i}\right) \right\Vert _{L^{n}}^{n}<\frac{%
1}{p_{1}^{n-1}}.  \label{eq2.10}
\end{equation}%
On the other hand,%
\begin{equation*}
\left\vert u_{i}\right\vert ^{\frac{n}{n-1}}\leq \left( \left\vert
v_{i}\right\vert +\left\vert u\right\vert \right) ^{\frac{n}{n-1}}\leq
\left( 1+\varepsilon \right) \left\vert v_{i}\right\vert ^{\frac{n}{n-1}%
}+c\left( \varepsilon \right) \left\vert u\right\vert ^{\frac{n}{n-1}},
\end{equation*}%
hence%
\begin{equation*}
e^{2^{-\frac{1}{n-1}}a_{n}\left\vert u_{i}\right\vert ^{\frac{n}{n-1}}}\leq
e^{2^{-\frac{1}{n-1}}\left( 1+\varepsilon \right) a_{n}\left\vert
v_{i}\right\vert ^{\frac{n}{n-1}}}e^{c\left( \varepsilon \right) \left\vert
u\right\vert ^{\frac{n}{n-1}}}.
\end{equation*}%
It follows from Lemma \ref{lem2.1} that%
\begin{eqnarray*}
\int_{B_{r}^{+}}e^{2^{-\frac{1}{n-1}}a_{n}p_{1}\left\vert v_{i}\right\vert ^{%
\frac{n}{n-1}}}d\mu &\leq &c\int_{B_{R}^{+}}e^{2^{-\frac{1}{n-1}%
}a_{n}p_{1}\left\vert \varphi v_{i}\right\vert ^{\frac{n}{n-1}}}dx \\
&\leq &c\int_{B_{R}^{+}}e^{2^{-\frac{1}{n-1}}a_{n}\frac{\left\vert \varphi
v_{i}\right\vert ^{\frac{n}{n-1}}}{\left\Vert \nabla \left( \varphi
v_{i}\right) \right\Vert _{L^{n}}^{\frac{n}{n-1}}}}dx \\
&\leq &c.
\end{eqnarray*}%
Using $p<\frac{p_{1}}{1+\varepsilon }$ and Lemma \ref{lem2.2}, it follows
from Holder's inequality that $e^{2^{-\frac{1}{n-1}}a_{n}\left\vert
u_{i}\right\vert ^{\frac{n}{n-1}}}$ is bounded in $L^{p}\left(
B_{r}^{+},d\mu \right) $.
\end{proof}

\begin{theorem}
\label{thm2.1}Let $M^{n}$ be a $C^{1}$ compact manifold with boundary and $g$
be a continuous Riemannian metric on $M$. If $E\subset M$ with $\mu \left(
E\right) \geq \delta >0$, $u\in W^{1,n}\left( M\right) \backslash \left\{
0\right\} $ with $u_{E}=0$, then for any $0<a<a_{n}=n\left\vert \mathbb{S}%
^{n-1}\right\vert ^{\frac{1}{n-1}}$,%
\begin{equation}
\int_{M}e^{2^{-\frac{1}{n-1}}a\frac{\left\vert u\right\vert ^{\frac{n}{n-1}}%
}{\left\Vert \nabla u\right\Vert _{L^{n}}^{\frac{n}{n-1}}}}d\mu \leq c\left(
a,\delta \right) <\infty .  \label{eq2.11}
\end{equation}
\end{theorem}

\begin{proof}
For any $v\in W^{1,n}\left( M\right) $,%
\begin{equation*}
\left\Vert v-v_{M}\right\Vert _{L^{n}}\leq c\left\Vert \nabla v\right\Vert
_{L^{n}}.
\end{equation*}%
On the other hand,%
\begin{equation*}
\left\Vert v-v_{E}\right\Vert _{L^{n}}\leq \frac{c}{\mu \left( E\right) ^{%
\frac{1}{n}}}\left\Vert v\right\Vert _{L^{n}}\leq c\left( \delta \right)
\left\Vert v\right\Vert _{L^{n}}.
\end{equation*}%
Replacing $v$ by $v-v_{M}$, we see%
\begin{equation*}
\left\Vert v-v_{E}\right\Vert _{L^{2}}\leq c\left( \delta \right) \left\Vert
v-v_{M}\right\Vert _{L^{n}}\leq c\left( \delta \right) \left\Vert \nabla
v\right\Vert _{L^{n}}.
\end{equation*}

If (\ref{eq2.11}) is not true, then we can find a sequence $u_{i}\in
W^{1,n}\left( M\right) $, $E_{i}\subset M$ with $\mu \left( E_{i}\right)
\geq \delta $, $u_{i,E_{i}}=0$, $\left\Vert \nabla u_{i}\right\Vert
_{L^{n}}=1$ and%
\begin{equation*}
\int_{M}e^{2^{-\frac{1}{n-1}}a\left\vert u_{i}\right\vert ^{\frac{n}{n-1}%
}}d\mu \rightarrow \infty
\end{equation*}%
as $i\rightarrow \infty $. Since%
\begin{equation*}
\left\Vert u_{i}\right\Vert _{L^{n}}=\left\Vert u_{i}-u_{i,E_{i}}\right\Vert
_{L^{n}}\leq c\left( \delta \right) \left\Vert \nabla u_{i}\right\Vert
_{L^{n}}=c\left( \delta \right) ,
\end{equation*}%
we see $u_{i}$ is bounded in $W^{1,n}\left( M\right) $. After passing to a
subsequence we can find $u\in W^{1,n}\left( M\right) $ such that $%
u_{i}\rightharpoonup u$ weakly in $W^{1,n}\left( M\right) $ and a measure on 
$M$, $\sigma $ such that%
\begin{equation*}
\left\vert \nabla u_{i}\right\vert ^{n}d\mu \rightarrow \left\vert \nabla
u\right\vert ^{n}d\mu +\sigma
\end{equation*}%
as measure. Note that $\sigma \left( M\right) \leq 1$. For any $x\in
M\backslash \partial M$, since%
\begin{equation*}
0<\frac{a}{a_{n}}<\sigma \left( \left\{ x\right\} \right) ^{-\frac{1}{n-1}},
\end{equation*}%
it follows from Proposition \ref{prop1.1} that for some $r>0$, we have%
\begin{equation*}
\sup_{i}\int_{B_{r}\left( x\right) }e^{a\left\vert u_{i}\right\vert ^{\frac{n%
}{n-1}}}d\mu <\infty .
\end{equation*}%
Hence%
\begin{equation*}
\sup_{i}\int_{B_{r}\left( x\right) }e^{2^{-\frac{1}{n-1}}a\left\vert
u_{i}\right\vert ^{\frac{n}{n-1}}}d\mu <\infty .
\end{equation*}%
For $x\in \partial M$, using $0<\frac{a}{a_{n}}<\sigma \left( \left\{
x\right\} \right) ^{-\frac{1}{n-1}}$ and Proposition \ref{prop2.1}, we can
find $r>0$ such that%
\begin{equation*}
\sup_{i}\int_{B_{r}\left( x\right) }e^{2^{-\frac{1}{n-1}}a\left\vert
u_{i}\right\vert ^{\frac{n}{n-1}}}d\mu <\infty .
\end{equation*}%
A covering argument implies%
\begin{equation*}
\sup_{i}\int_{M}e^{2^{-\frac{1}{n-1}}a\left\vert u_{i}\right\vert ^{\frac{n}{%
n-1}}}d\mu <\infty .
\end{equation*}%
This contradicts with the choice of $u_{i}$.
\end{proof}

\begin{corollary}
\label{cor2.1}Let $\Omega \subset \mathbb{R}^{n}$ be a bounded open domain
with $C^{1}$ boundary, $0<a<a_{n}$, then for any $u\in W^{1,n}\left( \Omega
\right) \backslash \left\{ 0\right\} $ with $\int_{\Omega }udx=0$, we have%
\begin{equation}
\int_{\Omega }e^{2^{-\frac{1}{n-1}}a\frac{\left\vert u\right\vert ^{\frac{n}{%
n-1}}}{\left\Vert \nabla u\right\Vert _{L^{n}}^{\frac{n}{n-1}}}}dx\leq
c\left( a,\Omega \right) <\infty .  \label{eq2.12}
\end{equation}
\end{corollary}

Corollary \ref{cor2.1} should be compared to \cite[Theorem 1.1]{Ci}.

\begin{example}
\label{ex2.1}Let $M^{n}$ be a $C^{1}$ compact manifold with boundary and $g$
be a continuous Riemannian metric on $M$. Denote%
\begin{equation}
\kappa \left( M,g\right) =\inf_{\substack{ u\in W^{1,n}\left( M\right)
\backslash \left\{ 0\right\}  \\ u_{M}=0}}\frac{\left\Vert \nabla
u\right\Vert _{L^{n}}}{\left\Vert u\right\Vert _{L^{n}}}.  \label{eq2.13}
\end{equation}%
Assume $0\leq \kappa <\kappa \left( M,g\right) $, $0<a<a_{n}$, $u\in
W^{1,n}\left( M\right) $ with $u_{M}=0$ and%
\begin{equation}
\left\Vert \nabla u\right\Vert _{L^{n}}^{n}-\kappa ^{n}\left\Vert
u\right\Vert _{L^{n}}^{n}\leq 1,  \label{eq2.14}
\end{equation}%
then we have%
\begin{equation}
\int_{M}e^{2^{-\frac{1}{n-1}}a\left\vert u\right\vert ^{\frac{n}{n-1}}}d\mu
\leq c\left( \kappa ,a\right) <\infty .  \label{eq2.15}
\end{equation}
\end{example}

Since the proof of Example \ref{ex2.1} is almost identical to the proof of
Example \ref{ex1.1} (using Proposition \ref{prop1.1} and \ref{prop2.1} when
necessary), we omit it here. Example \ref{ex2.1} should be compared with 
\cite{N1}.

Next we switch to the zero boundary value case. For $R>0$, we denote $\Sigma
_{R}$ as the base of $B_{R}^{+}$ i.e.%
\begin{equation}
\Sigma _{R}=\left\{ \left( x^{\prime },0\right) :x^{\prime }\in \mathbb{R}%
^{n-1},\left\vert x^{\prime }\right\vert <R\right\} .  \label{eq2.16}
\end{equation}

\begin{proposition}
\label{prop2.2}Let $0<R\leq 1$, $g$ be a continuous Riemannian metric on $%
\overline{B_{R}^{+}}$. Assume $u_{i}\in W^{1,n}\left( B_{R}^{+}\right) $
s.t. $\left. u_{i}\right\vert _{\Sigma _{R}}=0$, $u_{i}\rightharpoonup u$
weakly in $W^{1,n}\left( B_{R}^{+}\right) $ and%
\begin{equation*}
\left\vert \nabla _{g}u_{i}\right\vert _{g}^{n}d\mu \rightarrow \left\vert
\nabla _{g}u\right\vert _{g}^{n}d\mu +\sigma \text{ as measure on }B_{R}^{+}%
\text{.}
\end{equation*}%
If $0<p<\sigma \left( \left\{ 0\right\} \right) ^{-\frac{1}{n-1}}$, then
there exists $r>0$ such that%
\begin{equation}
\sup_{i}\int_{B_{r}^{+}}e^{a_{n}p\left\vert u_{i}\right\vert ^{\frac{n}{n-1}%
}}d\mu <\infty .  \label{eq2.17}
\end{equation}%
Here%
\begin{equation}
a_{n}=n\left\vert \mathbb{S}^{n-1}\right\vert ^{\frac{1}{n-1}}.
\label{eq2.18}
\end{equation}
\end{proposition}

\begin{proof}
Let $h$ be a continuous Riemannian metric on $\overline{B_{R}}$, which is an
extension of $g$. We define%
\begin{equation*}
v_{i}\left( x\right) =\left\{ 
\begin{tabular}{ll}
$u_{i}\left( x\right) ,$ & if $x\in B_{R}^{+}$ \\ 
$0,$ & if $x\in B_{R}\backslash B_{R}^{+}$%
\end{tabular}%
\right. ,\quad v\left( x\right) =\left\{ 
\begin{tabular}{ll}
$u\left( x\right) ,$ & if $x\in B_{R}^{+}$ \\ 
$0,$ & if $x\in B_{R}\backslash B_{R}^{+}$%
\end{tabular}%
\right. ,
\end{equation*}%
and a measure $\tau $ on $B_{R}$ by $\tau \left( E\right) =\sigma \left(
E\cap B_{R}\right) $ for any Borel set $E\subset B_{R}$. Then $v_{i},v\in
W^{1,n}\left( B_{R}\right) $, $v_{i}\rightharpoonup v$ weakly in $%
W^{1,n}\left( B_{R}\right) $ and%
\begin{equation*}
\left\vert \nabla _{h}v_{i}\right\vert _{h}^{n}d\mu _{h}\rightarrow
\left\vert \nabla _{h}v\right\vert _{h}^{n}d\mu _{h}+\tau \text{ as measure
on }B_{R}.
\end{equation*}%
Here $\mu _{h}$ is the measure associated with $h$. Since $\tau \left(
\left\{ 0\right\} \right) =\sigma \left( \left\{ 0\right\} \right) $, it
follows from Proposition \ref{prop2.1} that for some $r>0$,%
\begin{equation*}
\sup_{i}\int_{B_{r}}e^{a_{n}p\left\vert v_{i}\right\vert ^{\frac{n}{n-1}%
}}d\mu _{h}<\infty .
\end{equation*}%
Hence%
\begin{equation*}
\sup_{i}\int_{B_{r}^{+}}e^{a_{n}p\left\vert u_{i}\right\vert ^{\frac{n}{n-1}%
}}d\mu <\infty .
\end{equation*}
\end{proof}

\begin{theorem}
\label{thm2.2}Let $M^{n}$ be a $C^{1}$ compact manifold with boundary and $g$
be a continuous Riemannian metric on $M$. Denote%
\begin{equation}
\kappa _{0}\left( M,g\right) =\inf_{u\in W_{0}^{1,n}\left( M\right)
\backslash \left\{ 0\right\} }\frac{\left\Vert \nabla u\right\Vert _{L^{n}}}{%
\left\Vert u\right\Vert _{L^{n}}}.  \label{eq2.19}
\end{equation}%
Assume $0\leq \kappa <\kappa _{0}\left( M,g\right) $, $0<a<a_{n}=n\left\vert 
\mathbb{S}^{n-1}\right\vert ^{\frac{1}{n-1}}$, $u\in W_{0}^{1,n}\left(
M\right) $ and%
\begin{equation}
\left\Vert \nabla u\right\Vert _{L^{n}}^{n}-\kappa ^{n}\left\Vert
u\right\Vert _{L^{n}}^{n}\leq 1,  \label{eq2.20}
\end{equation}%
then we have%
\begin{equation}
\int_{M}e^{a\left\vert u\right\vert ^{\frac{n}{n-1}}}d\mu \leq c\left(
\kappa ,a\right) <\infty .  \label{eq2.21}
\end{equation}
\end{theorem}

Since the proof of Theorem \ref{thm2.2} is almost identical to the proof of
Example \ref{ex1.1} (using Proposition \ref{prop1.1} and \ref{prop2.2} when
necessary), we omit it here. Theorem \ref{thm2.2} should be compared with 
\cite{AD,N2,T,Y1}.

\section{Second order Sobolev spaces\label{sec3}}

Let $n\in \mathbb{N}$, $n\geq 3$ and $\Omega \subset \mathbb{R}^{n}$ be an
open domain. For $1\leq p<\infty $ and $u\in W^{2,p}\left( \Omega \right) $,
we denote%
\begin{equation}
\left\Vert u\right\Vert _{W^{2,p}\left( \Omega \right) }=\left[ \int_{\Omega
}\left( \left\vert u\right\vert ^{p}+\left\vert \nabla u\right\vert
^{p}+\left\vert D^{2}u\right\vert ^{p}\right) dx\right] ^{\frac{1}{p}}
\label{eq3.1}
\end{equation}%
and $W_{0}^{2,p}\left( \Omega \right) $ as the closure of $C_{c}^{\infty
}\left( \Omega \right) $ in $W^{2,p}\left( \Omega \right) $. For $u\in
W^{2,\infty }\left( \Omega \right) $, we denote%
\begin{equation}
\left\Vert u\right\Vert _{W^{2,\infty }\left( \Omega \right) }=\left\Vert
u\right\Vert _{L^{\infty }\left( \Omega \right) }+\left\Vert \nabla
u\right\Vert _{L^{\infty }\left( \Omega \right) }+\left\Vert
D^{2}u\right\Vert _{L^{\infty }\left( \Omega \right) }.  \label{eq3.2}
\end{equation}

For $R>0$, it is shown in \cite{A} that for any $u\in W_{0}^{2,\frac{n}{2}%
}\left( B_{R}^{n}\right) \backslash \left\{ 0\right\} $,%
\begin{equation}
\int_{B_{R}}\exp \left( a_{2,n}\frac{\left\vert u\right\vert ^{\frac{n}{n-2}}%
}{\left\Vert \Delta u\right\Vert _{L^{\frac{n}{2}}}^{\frac{n}{n-2}}}\right)
dx\leq c\left( n\right) R^{n}.  \label{eq3.3}
\end{equation}%
Here%
\begin{equation}
a_{2,n}=\frac{n}{\left\vert \mathbb{S}^{n-1}\right\vert }\left( \frac{4\pi ^{%
\frac{n}{2}}}{\Gamma \left( \frac{n-2}{2}\right) }\right) ^{\frac{n}{n-2}}
\label{eq3.4}
\end{equation}%
and%
\begin{equation}
\Gamma \left( \alpha \right) =\int_{0}^{\infty }t^{\alpha -1}e^{-t}dt
\label{eq3.5}
\end{equation}%
for $\alpha >0$.

Similar to Lemma \ref{lem1.1}, we have

\begin{lemma}
\label{lem3.1}If $u\in W^{2,\frac{n}{2}}\left( B_{R}^{n}\right) $, then for
any $a>0$,%
\begin{equation}
\int_{B_{R}}e^{a\left\vert u\right\vert ^{\frac{n}{n-2}}}dx<\infty .
\label{eq3.6}
\end{equation}
\end{lemma}

\begin{proof}
First we assume $u\in W_{0}^{2,\frac{n}{2}}\left( B_{R}\right) $. Without
losing of generality, we can assume $u$ is unbounded. For $\varepsilon >0$,
a tiny number to be determined, we can find $v\in C_{c}^{\infty }\left(
B_{R}\right) $ such that%
\begin{equation*}
\left\Vert u-v\right\Vert _{W^{2,\frac{n}{2}}}<\varepsilon .
\end{equation*}%
Hence%
\begin{equation*}
\left\Vert \Delta u-\Delta v\right\Vert _{L^{\frac{n}{2}}}\leq c\left(
n\right) \varepsilon .
\end{equation*}%
Let $w=u-v$, then%
\begin{equation*}
\left\vert u\right\vert =\left\vert v+w\right\vert \leq \left\Vert
v\right\Vert _{L^{\infty }}+\left\vert w\right\vert .
\end{equation*}%
Hence%
\begin{equation*}
\left\vert u\right\vert ^{\frac{n}{n-2}}\leq 2^{\frac{2}{n-2}}\left\Vert
v\right\Vert _{L^{\infty }}^{\frac{n}{n-2}}+2^{\frac{2}{n-2}}\left\vert
w\right\vert ^{\frac{n}{n-2}}.
\end{equation*}%
It follows that%
\begin{equation*}
e^{a\left\vert u\right\vert ^{\frac{n}{n-1}}}\leq e^{2^{\frac{2}{n-2}%
}a\left\Vert v\right\Vert _{L^{\infty }}^{\frac{n}{n-2}}}e^{2^{\frac{2}{n-2}%
}a\left\vert w\right\vert ^{\frac{n}{n-2}}}\leq e^{2^{\frac{2}{n-2}%
}a\left\Vert v\right\Vert _{L^{\infty }}^{\frac{n}{n-2}}}e^{a_{2,n}\frac{%
\left\vert w\right\vert ^{\frac{n}{n-2}}}{\left\Vert \Delta w\right\Vert
_{L^{\frac{n}{2}}}^{\frac{n}{n-2}}}}
\end{equation*}%
if $\varepsilon $ is small enough. Hence%
\begin{equation*}
\int_{B_{R}}e^{a\left\vert u\right\vert ^{\frac{n}{n-2}}}dx\leq c\left(
n\right) e^{2^{\frac{2}{n-2}}a\left\Vert v\right\Vert _{L^{\infty }}^{\frac{n%
}{n-2}}}R^{n}<\infty .
\end{equation*}%
In general, if $u\in W^{2,\frac{n}{2}}\left( B_{R}\right) $, then we can
find $\widetilde{u}\in W_{0}^{2,\frac{n}{2}}\left( B_{2R}\right) $ such that 
$\left. \widetilde{u}\right\vert _{B_{R}}=u$. Hence%
\begin{equation*}
\int_{B_{R}}e^{a\left\vert u\right\vert ^{\frac{n}{n-2}}}dx\leq
\int_{B_{2R}}e^{a\left\vert \widetilde{u}\right\vert ^{\frac{n}{n-2}%
}}dx<\infty .
\end{equation*}
\end{proof}

\begin{proposition}
\label{prop3.1}Let $0<R\leq 1$, $g$ be a $C^{1}$ Riemannian metric on $%
\overline{B_{R}^{n}}$. Assume $u_{i}\in W^{2,\frac{n}{2}}\left(
B_{R}^{n}\right) $, $u_{i}\rightharpoonup u$ weakly in $W^{2,\frac{n}{2}%
}\left( B_{R}\right) $ and%
\begin{equation*}
\left\vert \Delta _{g}u_{i}\right\vert ^{\frac{n}{2}}d\mu \rightarrow
\left\vert \Delta _{g}u\right\vert ^{\frac{n}{2}}d\mu +\sigma \text{ as
measure on }B_{R}.
\end{equation*}%
If $0<p<\sigma \left( \left\{ 0\right\} \right) ^{-\frac{2}{n-2}}$, then for
some $r>0$,%
\begin{equation}
\sup_{i}\int_{B_{r}}e^{a_{2,n}p\left\vert u_{i}\right\vert ^{\frac{n}{n-2}%
}}d\mu <\infty .  \label{eq3.7}
\end{equation}%
Here%
\begin{equation}
a_{2,n}=\frac{n}{\left\vert \mathbb{S}^{n-1}\right\vert }\left( \frac{4\pi ^{%
\frac{n}{2}}}{\Gamma \left( \frac{n-2}{2}\right) }\right) ^{\frac{n}{n-2}}
\label{eq3.8}
\end{equation}
\end{proposition}

By a linear changing of variable and shrinking $R$ if necessary we can
assume $g=g_{ij}dx_{i}dx_{j}$ with $g_{ij}\left( 0\right) =\delta _{ij}$. Let%
\begin{equation*}
G=\det \left[ g_{ij}\right] _{1\leq i,j\leq n};\quad \left[ g^{ij}\right]
_{1\leq i,j\leq n}=\left[ g_{ij}\right] _{1\leq i,j\leq n}^{-1}.
\end{equation*}%
If $v$ is a function on $B_{R}$, then%
\begin{eqnarray*}
\Delta _{g}v &=&\frac{1}{\sqrt{G}}\partial _{i}\left( \sqrt{G}g^{ij}\partial
_{j}v\right)  \\
&=&g^{ij}\partial _{ij}v+\partial _{i}g^{ij}\partial _{j}v+g^{ij}\partial
_{i}\log \sqrt{G}\partial _{j}v.
\end{eqnarray*}%
For $v\in W_{0}^{2,\frac{n}{2}}\left( B_{R}\right) $, then standard elliptic
theory (see \cite{GiT}) tells us%
\begin{equation}
\left\Vert v\right\Vert _{W^{2,\frac{n}{2}}\left( B_{R}\right) }\leq
c_{1}\left\Vert \Delta _{g}v\right\Vert _{L^{\frac{n}{2}}\left( B_{R}\right)
}.  \label{eq3.9}
\end{equation}

On the other hand, let $0<R_{1}<R$ be a small number and $v\in W_{0}^{2,%
\frac{n}{2}}\left( B_{R_{1}}\right) $, we have%
\begin{eqnarray*}
\Delta v &=&\left( \delta _{ij}-g^{ij}\right) \partial _{ij}v+g^{ij}\partial
_{ij}v \\
&=&\Delta _{g}v+\left( \delta _{ij}-g^{ij}\right) \partial _{ij}v-\partial
_{i}g^{ij}\partial _{j}v-g^{ij}\partial _{i}\log \sqrt{G}\partial _{j}v.
\end{eqnarray*}%
It follows that%
\begin{equation}
\left\vert \Delta v\right\vert \leq \left\vert \Delta _{g}v\right\vert
+\varepsilon _{1}\left\vert D^{2}v\right\vert +c\left\vert \nabla
v\right\vert .  \label{eq3.10}
\end{equation}%
Here $\varepsilon _{1}=\varepsilon _{1}\left( R_{1}\right) >0$. Moreover $%
\varepsilon _{1}\rightarrow 0$ as $R_{1}\rightarrow 0^{+}$. We have%
\begin{eqnarray}
&&\left\Vert \Delta v\right\Vert _{L^{\frac{n}{2}}\left( B_{R}\right) }^{%
\frac{n}{2}}  \label{eq3.11} \\
&\leq &\left( \left\Vert \Delta _{g}v\right\Vert _{L^{\frac{n}{2}%
}}+\varepsilon _{1}\left\Vert D^{2}v\right\Vert _{L^{\frac{n}{2}%
}}+c\left\Vert \nabla v\right\Vert _{L^{\frac{n}{2}}}\right) ^{\frac{n}{2}} 
\notag \\
&\leq &\left[ \left( 1+c_{1}\varepsilon _{1}\right) \left\Vert \Delta
_{g}v\right\Vert _{L^{\frac{n}{2}}}+c\left\Vert \nabla v\right\Vert _{L^{%
\frac{n}{2}}}\right] ^{\frac{n}{2}}  \notag \\
&\leq &\left( 1+c\varepsilon _{1}\right) \left\Vert \Delta _{g}v\right\Vert
_{L^{\frac{n}{2}}}^{\frac{n}{2}}+c\left( \varepsilon _{1}\right) \left\Vert
\nabla v\right\Vert _{L^{\frac{n}{2}}}^{\frac{n}{2}}.  \notag
\end{eqnarray}

To continue, we fix $p_{1}\in \left( p,\sigma \left( \left\{ 0\right\}
\right) ^{-\frac{2}{n-2}}\right) $, then%
\begin{equation}
\sigma \left( \left\{ 0\right\} \right) <\frac{1}{p_{1}^{\frac{n-2}{2}}}.
\label{eq3.12}
\end{equation}%
We can find $\varepsilon >0$ s.t.%
\begin{equation}
\left( 1+\varepsilon \right) \sigma \left( \left\{ 0\right\} \right) <\frac{1%
}{p_{1}^{\frac{n-2}{2}}}  \label{eq3.13}
\end{equation}%
and%
\begin{equation}
\left( 1+\varepsilon \right) p<p_{1}.  \label{eq3.14}
\end{equation}%
Let $v_{i}=u_{i}-u$, then $v_{i}\rightharpoonup 0$ weakly in $W^{2,\frac{n}{2%
}}\left( B_{R}\right) $ and $v_{i}\rightarrow 0$ in $W^{1,\frac{n}{2}}\left(
B_{R}\right) $. Let $0<R_{1}<R$ be a small number to be determined. For $%
\varphi \in C_{c}^{\infty }\left( B_{R_{1}}\right) $, by the previous
discussion we have%
\begin{eqnarray*}
&&\left\Vert \Delta \left( \varphi v_{i}\right) \right\Vert _{L^{\frac{n}{2}%
}\left( B_{R}\right) }^{\frac{n}{2}} \\
&\leq &\left( 1+c\varepsilon _{1}\right) \left\Vert \Delta _{g}\left(
\varphi v_{i}\right) \right\Vert _{L^{\frac{n}{2}}}^{\frac{n}{2}}+c\left(
\varepsilon _{1}\right) \left\Vert \nabla \left( \varphi v_{i}\right)
\right\Vert _{L^{\frac{n}{2}}}^{\frac{n}{2}} \\
&\leq &\left( 1+c\varepsilon _{1}\right) \left( \left\Vert \varphi \Delta
_{g}v_{i}\right\Vert _{L^{\frac{n}{2}}}+c\left\Vert \varphi \right\Vert
_{W^{2,\infty }}\left\Vert v_{i}\right\Vert _{W^{1,\frac{n}{2}}}\right) ^{%
\frac{n}{2}}+c\left( \varepsilon _{1}\right) \left\Vert \nabla \left(
\varphi v_{i}\right) \right\Vert _{L^{\frac{n}{2}}}^{\frac{n}{2}} \\
&\leq &\left( 1+\frac{\varepsilon }{4}\right) \left\Vert \varphi \Delta
_{g}v_{i}\right\Vert _{L^{\frac{n}{2}}}^{\frac{n}{2}}+c\left( \varepsilon
\right) \left\Vert \varphi \right\Vert _{W^{2,\infty }}^{\frac{n}{2}%
}\left\Vert v_{i}\right\Vert _{W^{1,\frac{n}{2}}}^{\frac{n}{2}} \\
&\leq &\left( 1+\frac{\varepsilon }{4}\right) \left( \left\Vert \varphi
\Delta _{g}u_{i}\right\Vert _{L^{\frac{n}{2}}}+\left\Vert \varphi \Delta
_{g}u\right\Vert _{L^{\frac{n}{2}}}\right) ^{\frac{n}{2}}+c\left(
\varepsilon \right) \left\Vert \varphi \right\Vert _{W^{2,\infty }}^{\frac{n%
}{2}}\left\Vert v_{i}\right\Vert _{W^{1,\frac{n}{2}}}^{\frac{n}{2}} \\
&\leq &\left( 1+\frac{\varepsilon }{2}\right) \left\Vert \varphi \Delta
_{g}u_{i}\right\Vert _{L^{\frac{n}{2}}}^{\frac{n}{2}}+c\left( \varepsilon
\right) \left\Vert \varphi \Delta _{g}u\right\Vert _{L^{\frac{n}{2}}}^{\frac{%
n}{2}}+c\left( \varepsilon \right) \left\Vert \varphi \right\Vert
_{W^{2,\infty }}^{\frac{n}{2}}\left\Vert v_{i}\right\Vert _{W^{1,\frac{n}{2}%
}}^{\frac{n}{2}} \\
&\leq &\left( 1+\varepsilon \right) \left\Vert \varphi \Delta
_{g}u_{i}\right\Vert _{L^{\frac{n}{2}}\left( B_{R},d\mu \right) }^{\frac{n}{2%
}}+c\left( \varepsilon \right) \left\Vert \varphi \Delta _{g}u\right\Vert
_{L^{\frac{n}{2}}\left( B_{R},d\mu \right) }^{\frac{n}{2}}+c\left(
\varepsilon \right) \left\Vert \varphi \right\Vert _{W^{2,\infty }}^{\frac{n%
}{2}}\left\Vert v_{i}\right\Vert _{W^{1,\frac{n}{2}}}^{\frac{n}{2}}
\end{eqnarray*}%
if $R_{1}$ is small enough. Hence%
\begin{eqnarray*}
&&\lim \sup_{i\rightarrow \infty }\left\Vert \Delta \left( \varphi
v_{i}\right) \right\Vert _{L^{\frac{n}{2}}\left( B_{R}\right) }^{\frac{n}{2}}
\\
&\leq &\left( 1+\varepsilon \right) \int_{B_{R}}\left\vert \varphi
\right\vert ^{\frac{n}{2}}d\sigma +\left( 1+\varepsilon \right)
\int_{B_{R}}\left\vert \varphi \right\vert ^{\frac{n}{2}}\left\vert \Delta
_{g}u\right\vert ^{\frac{n}{2}}d\mu +c\left( \varepsilon \right) \left\Vert
\varphi \Delta _{g}u\right\Vert _{L^{\frac{n}{2}}\left( B_{R},d\mu \right)
}^{\frac{n}{2}} \\
&\leq &\left( 1+\varepsilon \right) \int_{B_{R}}\left\vert \varphi
\right\vert ^{\frac{n}{2}}d\sigma +c\left( \varepsilon \right)
\int_{B_{R}}\left\vert \varphi \right\vert ^{\frac{n}{2}}\left\vert \Delta
_{g}u\right\vert ^{\frac{n}{2}}d\mu .
\end{eqnarray*}%
Since $\left( 1+\varepsilon \right) \sigma \left( \left\{ 0\right\} \right) <%
\frac{1}{p_{1}^{\frac{n-2}{2}}}$, we can choose $\varphi \in C_{c}^{\infty
}\left( B_{R_{1}}\right) $ such that $\left. \varphi \right\vert _{B_{r}}=1$
for some $r>0$ and%
\begin{equation}
\left( 1+\varepsilon \right) \int_{B_{R}}\left\vert \varphi \right\vert ^{%
\frac{n}{2}}d\sigma +c\left( \varepsilon \right) \int_{B_{R}}\left\vert
\varphi \right\vert ^{\frac{n}{2}}\left\vert \Delta _{g}u\right\vert ^{\frac{%
n}{2}}d\mu <\frac{1}{p_{1}^{\frac{n-2}{2}}}.  \label{eq3.15}
\end{equation}%
Hence for $i$ large enough, we have%
\begin{equation}
\left\Vert \Delta \left( \varphi v_{i}\right) \right\Vert _{L^{\frac{n}{2}%
}\left( B_{R}\right) }^{\frac{n}{2}}<\frac{1}{p_{1}^{\frac{n-2}{2}}}.
\label{eq3.16}
\end{equation}%
This implies%
\begin{equation}
\left\Vert \Delta \left( \varphi v_{i}\right) \right\Vert _{L^{\frac{n}{2}%
}\left( B_{R}\right) }^{\frac{n}{n-2}}<\frac{1}{p_{1}}.  \label{eq3.17}
\end{equation}%
On the other hand,%
\begin{equation*}
\left\vert u_{i}\right\vert ^{\frac{n}{n-2}}\leq \left( \left\vert
v_{i}\right\vert +\left\vert u\right\vert \right) ^{\frac{n}{n-2}}\leq
\left( 1+\varepsilon \right) \left\vert v_{i}\right\vert ^{\frac{n}{n-2}%
}+c\left( \varepsilon \right) \left\vert u\right\vert ^{\frac{n}{n-2}},
\end{equation*}%
hence%
\begin{equation*}
e^{a_{2,n}\left\vert u_{i}\right\vert ^{\frac{n}{n-2}}}\leq e^{\left(
1+\varepsilon \right) a_{2,n}\left\vert v_{i}\right\vert ^{\frac{n}{n-2}%
}}e^{c\left( \varepsilon \right) \left\vert u\right\vert ^{\frac{n}{n-2}}}.
\end{equation*}%
It follows from (\ref{eq3.3}) that%
\begin{equation*}
\int_{B_{r}}e^{a_{2,n}p_{1}\left\vert v_{i}\right\vert ^{\frac{n}{n-2}}}d\mu
\leq c\int_{B_{R}}e^{a_{2,n}p_{1}\left\vert \varphi v_{i}\right\vert ^{\frac{%
n}{n-2}}}dx\leq c\int_{B_{R}}e^{a_{2,n}\frac{\left\vert \varphi
v_{i}\right\vert ^{\frac{n}{n-2}}}{\left\Vert \nabla \left( \varphi
v_{i}\right) \right\Vert _{L^{n}}^{\frac{n}{n-2}}}}dx\leq c.
\end{equation*}%
Using $p<\frac{p_{1}}{1+\varepsilon }$ and Lemma \ref{lem3.1}, it follows
from Holder's inequality that $e^{a_{2,n}\left\vert u_{i}\right\vert ^{\frac{%
n}{n-2}}}$ is bounded in $L^{p}\left( B_{r},d\mu \right) $. This finishes
the proof of Proposition \ref{prop3.1}.

\begin{theorem}
\label{thm3.1}Let $M^{n}$ be a $C^{2}$ compact manifold with a $C^{1}$
Riemannian metric $g$. If $E\subset M$ with $\mu \left( E\right) \geq \delta
>0$, $u\in W^{2,\frac{n}{2}}\left( M\right) \backslash \left\{ 0\right\} $
with $u_{E}=0$, then for any $0<a<a_{2,n}=\frac{n}{\left\vert \mathbb{S}%
^{n-1}\right\vert }\left( \frac{4\pi ^{\frac{n}{2}}}{\Gamma \left( \frac{n-2%
}{2}\right) }\right) ^{\frac{n}{n-2}}$,%
\begin{equation}
\int_{M}e^{a\frac{\left\vert u\right\vert ^{\frac{n}{n-2}}}{\left\Vert
\Delta u\right\Vert _{L^{\frac{n}{2}}}^{\frac{n}{n-2}}}}d\mu \leq c\left(
a,\delta \right) <\infty .  \label{eq3.18}
\end{equation}
\end{theorem}

\begin{proof}
For any $v\in W^{2,\frac{n}{2}}\left( M\right) $, standard elliptic theory
(see \cite{GiT}) and compactness argument tells us%
\begin{equation*}
\left\Vert v-v_{M}\right\Vert _{L^{\frac{n}{2}}}\leq c\left\Vert \Delta
v\right\Vert _{L^{\frac{n}{2}}}.
\end{equation*}%
On the other hand,%
\begin{equation*}
\left\Vert v-v_{E}\right\Vert _{L^{\frac{n}{2}}}\leq \frac{c}{\mu \left(
E\right) ^{\frac{2}{n}}}\left\Vert v\right\Vert _{L^{\frac{n}{2}}}\leq
c\left( \delta \right) \left\Vert v\right\Vert _{L^{\frac{n}{2}}}.
\end{equation*}%
Replacing $v$ by $v-v_{M}$, we see%
\begin{equation*}
\left\Vert v-v_{E}\right\Vert _{L^{\frac{n}{2}}}\leq c\left( \delta \right)
\left\Vert v-v_{M}\right\Vert _{L^{\frac{n}{2}}}\leq c\left( \delta \right)
\left\Vert \Delta v\right\Vert _{L^{\frac{n}{2}}}.
\end{equation*}

If (\ref{eq3.18}) is not true, then we can find a sequence $u_{i}\in W^{2,%
\frac{n}{2}}\left( M\right) $, $E_{i}\subset M$ with $\mu \left(
E_{i}\right) \geq \delta $, $u_{i,E_{i}}=0$, $\left\Vert \Delta
u_{i}\right\Vert _{L^{\frac{n}{2}}}=1$ and%
\begin{equation*}
\int_{M}e^{a\left\vert u_{i}\right\vert ^{\frac{n}{n-2}}}d\mu \rightarrow
\infty
\end{equation*}%
as $i\rightarrow \infty $. Since%
\begin{equation*}
\left\Vert u_{i}\right\Vert _{L^{\frac{n}{2}}}=\left\Vert
u_{i}-u_{i,E_{i}}\right\Vert _{L^{\frac{n}{2}}}\leq c\left( \delta \right)
\left\Vert \Delta u_{i}\right\Vert _{L^{\frac{n}{2}}}=c\left( \delta \right)
,
\end{equation*}%
standard elliptic estimate (see \cite{GiT}) gives us%
\begin{equation*}
\left\Vert u_{i}\right\Vert _{W^{2,\frac{n}{2}}}\leq c\left( \left\Vert
\Delta u_{i}\right\Vert _{L^{\frac{n}{2}}}+\left\Vert u_{i}\right\Vert _{L^{%
\frac{n}{2}}}\right) \leq c\left( \delta \right) .
\end{equation*}%
Hence $u_{i}$ is bounded in $W^{2,\frac{n}{2}}\left( M\right) $. After
passing to a subsequence we can find $u\in W^{2,\frac{n}{2}}\left( M\right) $
such that $u_{i}\rightharpoonup u$ weakly in $W^{2,\frac{n}{2}}\left(
M\right) $ and a measure on $M$, $\sigma $ such that%
\begin{equation*}
\left\vert \Delta u_{i}\right\vert ^{\frac{n}{2}}d\mu \rightarrow \left\vert
\Delta u\right\vert ^{\frac{n}{2}}d\mu +\sigma
\end{equation*}%
as measure. Note that $\sigma \left( M\right) \leq 1$. For any $x\in M$,
since%
\begin{equation*}
0<\frac{a}{a_{2,n}}<\sigma \left( \left\{ x\right\} \right) ^{-\frac{2}{n-2}%
},
\end{equation*}%
it follows from Proposition \ref{prop3.1} that for some $r>0$, we have%
\begin{equation*}
\sup_{i}\int_{B_{r}\left( x\right) }e^{a\left\vert u_{i}\right\vert ^{\frac{n%
}{n-2}}}d\mu <\infty .
\end{equation*}%
A covering argument implies%
\begin{equation*}
\sup_{i}\int_{M}e^{a\left\vert u_{i}\right\vert ^{\frac{n}{n-2}}}d\mu
<\infty .
\end{equation*}%
This contradicts with the choice of $u_{i}$.
\end{proof}

\begin{example}
\label{ex3.1}Let $M^{n}$ be a $C^{2}$ compact manifold with a $C^{1}$
Riemannian metric $g$. Denote%
\begin{equation}
\kappa _{2}\left( M,g\right) =\inf_{\substack{ u\in W^{2,\frac{n}{2}}\left(
M\right) \backslash \left\{ 0\right\}  \\ u_{M}=0}}\frac{\left\Vert \Delta
u\right\Vert _{L^{\frac{n}{2}}}}{\left\Vert u\right\Vert _{L^{\frac{n}{2}}}}.
\label{eq3.19}
\end{equation}%
Here $u_{M}=\frac{1}{\mu \left( M\right) }\int_{M}ud\mu $. Assume $0\leq
\kappa <\kappa _{2}\left( M,g\right) $, $0<a<a_{2,n}$, $u\in W^{2,\frac{n}{2}%
}\left( M\right) $ with $u_{M}=0$ and%
\begin{equation}
\left\Vert \Delta u\right\Vert _{L^{\frac{n}{2}}}^{\frac{n}{2}}-\kappa ^{%
\frac{n}{2}}\left\Vert u\right\Vert _{L^{\frac{n}{2}}}^{\frac{n}{2}}\leq 1,
\label{eq3.20}
\end{equation}%
then we have%
\begin{equation}
\int_{M}e^{a\left\vert u\right\vert ^{\frac{n}{n-2}}}d\mu \leq c\left(
\kappa ,a\right) <\infty .  \label{eq3.21}
\end{equation}
\end{example}

Since the proof of Example \ref{ex3.1} is almost identical to the proof of
Example \ref{ex1.1} (using Proposition \ref{prop3.1} when necessary), we
omit it here.

\subsection{Paneitz operator and Q curvature in dimension 4\label{sec3.1}}

Let $\left( M^{4},g\right) $ be a smooth compact Riemannian manifold with
dimension 4. The Paneitz operator is given by (see \cite{CY2, HY2})%
\begin{equation}
Pu=\Delta ^{2}u+2\func{div}\left( Rc\left( \nabla u,e_{i}\right)
e_{i}\right) -\frac{2}{3}\func{div}\left( R\nabla u\right) .  \label{eq3.22}
\end{equation}%
Here $e_{1},e_{2},e_{3},e_{4}$ is a local orthonormal frame with respect to $%
g$. The associated $Q$ curvature is%
\begin{equation}
Q=-\frac{1}{6}\Delta R-\frac{1}{2}\left\vert Rc\right\vert ^{2}+\frac{1}{6}%
R^{2}.  \label{eq3.23}
\end{equation}%
In 4-dimensional conformal geometry, $P$ and $Q$ play the same roles as $%
-\Delta $ and Gauss curvature in 2-dimensional conformal geometry.

For $u\in C^{\infty }\left( M\right) $, let%
\begin{eqnarray}
E\left( u\right) &=&\int_{M}Pu\cdot ud\mu  \label{eq3.24} \\
&=&\int_{M}\left( \left( \Delta u\right) ^{2}-2Rc\left( \nabla u,\nabla
u\right) +\frac{2}{3}R\left\vert \nabla u\right\vert ^{2}\right) d\mu . 
\notag
\end{eqnarray}%
By this formula, we know $E\left( u\right) $ still makes sense for $u\in
H^{2}\left( M\right) =W^{2,2}\left( M\right) $.

On the standard $\mathbb{S}^{4}$, $P\geq 0$ and $\ker P=\left\{ \text{%
constant functions}\right\} $. Moreover in \cite{Gu, GuV}, some general
criterion for such positivity condition to be valid were derived. If $P\geq
0 $ and $\ker P=\left\{ \text{constant functions}\right\} $, then for any $%
u\in H^{2}\left( M\right) $ with $u_{M}=0$,%
\begin{equation}
\left\Vert u\right\Vert _{L^{2}}^{2}\leq c\left( M,g\right) E\left( u\right)
.  \label{eq3.25}
\end{equation}%
On the other hand, standard elliptic theory (see \cite{GiT}) tells us%
\begin{eqnarray*}
\left\Vert u\right\Vert _{H^{2}}^{2} &\leq &c\left( M,g\right) \left(
\left\Vert \Delta u\right\Vert _{L^{2}}^{2}+\left\Vert u\right\Vert
_{L^{2}}^{2}\right) \\
&\leq &c\left( M,g\right) \left( E\left( u\right) +\left\Vert u\right\Vert
_{H^{1}}^{2}\right) \\
&\leq &c\left( M,g\right) E\left( u\right) +\frac{1}{2}\left\Vert
u\right\Vert _{H^{2}}^{2}+c\left( M,g\right) \left\Vert u\right\Vert
_{L^{2}}^{2} \\
&\leq &\frac{1}{2}\left\Vert u\right\Vert _{H^{2}}^{2}+c\left( M,g\right)
E\left( u\right) .
\end{eqnarray*}%
We have used the interpolation inequality in between. It follows that%
\begin{equation}
\left\Vert u\right\Vert _{H^{2}}^{2}\leq c\left( M,g\right) E\left( u\right)
.  \label{eq3.26}
\end{equation}

\begin{lemma}
\label{lem3.2}If $P\geq 0$, $\ker P=\left\{ \text{constant functions}%
\right\} $, $u\in H^{2}\left( M\right) \backslash \left\{ 0\right\} $ such
that $u_{M}=0$, then for any $a\in \left( 0,32\pi ^{2}\right) $,%
\begin{equation}
\int_{M}e^{a\frac{u^{2}}{E\left( u\right) }}d\mu \leq c\left( a\right)
<\infty .  \label{eq3.27}
\end{equation}%
In particular,%
\begin{equation}
\log \int_{M}e^{4u}d\mu \leq \frac{4}{a}E\left( u\right) +c\left( a\right) .
\label{eq3.28}
\end{equation}
\end{lemma}

Note that $a_{2,4}=32\pi ^{2}$. In \cite[Lemma 1.6]{CY2}, it is shown that
under the same assumption as Lemma \ref{lem3.2},%
\begin{equation}
\int_{M}e^{32\pi ^{2}\frac{u^{2}}{E\left( u\right) }}d\mu \leq c\left(
M,g\right) <\infty .  \label{eq3.29}
\end{equation}%
The argument is based on asymptotic expansion formula of Green's function of 
$\sqrt{P}$ and modified Adams' argument \cite{A}. Our approach below is more
elementary. Moreover, Lemma \ref{lem3.2} is sufficient for application to Q
curvature equation in \cite[Theorem 1.2]{CY2}, which is for the case $\left(
M,g\right) $ not conformal diffeomorphic to the standard $\mathbb{S}^{4}$
(see \cite{Gu} and \cite[Proposition 1.3]{HY1}).

\begin{proof}[Proof of Lemma \protect\ref{lem3.2}]
If (\ref{eq3.27}) is not true, then there exists $u_{i}\in H^{2}\left(
M\right) $ such that $u_{i,M}=0$, $E\left( u_{i}\right) =1$ and%
\begin{equation*}
\int_{M}e^{au_{i}^{2}}d\mu \rightarrow \infty
\end{equation*}%
as $i\rightarrow \infty $. By previous discussion we see%
\begin{equation*}
\left\Vert u_{i}\right\Vert _{H^{2}}^{2}\leq c\left( M,g\right) E\left(
u_{i}\right) =c\left( M,g\right) .
\end{equation*}%
Hence after passing to a subsequence, we can find a $u\in H^{2}\left(
M\right) $ and a measure $\sigma $ on $M$ such that $u_{i}\rightharpoonup u$
weakly in $H^{2}\left( M\right) $, $u_{i}\rightarrow u$ in $H^{1}\left(
M\right) $ and%
\begin{equation*}
\left( \Delta u_{i}\right) ^{2}d\mu \rightarrow \left( \Delta u\right)
^{2}d\mu +\sigma \text{ as measure.}
\end{equation*}%
Note that%
\begin{eqnarray*}
E\left( u_{i}\right) &=&\int_{M}\left( \left( \Delta u_{i}\right)
^{2}-2Rc\left( \nabla u_{i},\nabla u_{i}\right) +\frac{2}{3}R\left\vert
\nabla u_{i}\right\vert ^{2}\right) d\mu \\
&\rightarrow &E\left( u\right) +\sigma \left( M\right) ,
\end{eqnarray*}%
we see $E\left( u\right) +\sigma \left( M\right) =1$. Using the fact $%
E\left( u\right) \geq 0$, we see $\sigma \left( M\right) \leq 1$. For any $%
x\in M$, because%
\begin{equation*}
0<\frac{a}{32\pi ^{2}}<\sigma \left( \left\{ x\right\} \right) ^{-1},
\end{equation*}%
it follows from Proposition \ref{prop3.1} that for some $r>0$,%
\begin{equation*}
\sup_{i}\int_{B_{r}\left( x\right) }e^{au_{i}^{2}}d\mu <\infty .
\end{equation*}%
A covering arguments implies%
\begin{equation*}
\sup_{i}\int_{M}e^{au_{i}^{2}}d\mu <\infty .
\end{equation*}%
This contradicts with the choice of $u_{i}$.

Since%
\begin{equation*}
4u\leq a\frac{u^{2}}{E\left( u\right) }+\frac{4E\left( u\right) }{a},
\end{equation*}%
we see%
\begin{equation*}
e^{4u}\leq e^{\frac{4E\left( u\right) }{a}}e^{a\frac{u^{2}}{E\left( u\right) 
}}.
\end{equation*}%
Hence%
\begin{equation*}
\int_{M}e^{4u}d\mu \leq c\left( a\right) e^{\frac{4E\left( u\right) }{a}}
\end{equation*}%
and%
\begin{equation*}
\log \int_{M}e^{4u}d\mu \leq \frac{4}{a}E\left( u\right) +c\left( a\right) .
\end{equation*}
\end{proof}

\subsection{Functions on manifolds with nonempty boundary\label{sec3.2}}

We will use the same notations as in Section \ref{sec2}.

\begin{lemma}
\label{lem3.3}Assume $u\in W^{2,\frac{n}{2}}\left( B_{R}^{+}\right)
\backslash \left\{ 0\right\} $ such that $u\left( x\right) =0$ for $%
\left\vert x\right\vert $ close to $R$ and $\partial _{n}u=0$ on $\Sigma
_{R} $, then%
\begin{equation}
\int_{B_{R}^{+}}e^{2^{-\frac{2}{n-2}}a_{2,n}\frac{\left\vert u\right\vert ^{%
\frac{n}{n-2}}}{\left\Vert \Delta u\right\Vert _{L^{\frac{n}{2}}\left(
B_{R}^{+}\right) }^{\frac{n}{n-2}}}}dx\leq c\left( n\right) R^{n}.
\label{eq3.30}
\end{equation}
\end{lemma}

\begin{proof}
For $\left\vert x\right\vert <R$, we define%
\begin{equation*}
v\left( x\right) =\left\{ 
\begin{array}{cc}
u\left( x\right) , & \text{if }x_{n}>0; \\ 
u\left( x^{\prime },-x_{n}\right) , & \text{if }x_{n}<0\text{.}%
\end{array}%
\right.
\end{equation*}%
Then $v\in W_{0}^{2,\frac{n}{2}}\left( B_{R}\right) $ with $\left\Vert
\Delta v\right\Vert _{L^{\frac{n}{2}}\left( B_{R}\right) }^{\frac{n}{2}%
}=2\left\Vert \Delta u\right\Vert _{L^{\frac{n}{2}}\left( B_{R}^{+}\right)
}^{\frac{n}{2}}$. Hence%
\begin{eqnarray*}
\int_{B_{R}^{+}}e^{2^{-\frac{2}{n-2}}a_{2,n}\frac{\left\vert u\right\vert ^{%
\frac{n}{n-2}}}{\left\Vert \Delta u\right\Vert _{L^{\frac{n}{2}}\left(
B_{R}^{+}\right) }^{\frac{n}{n-2}}}}dx &=&\int_{B_{R}^{+}}e^{a_{2,n}\frac{%
\left\vert u\right\vert ^{\frac{n}{n-2}}}{\left\Vert \Delta v\right\Vert
_{L^{\frac{n}{2}}\left( B_{R}\right) }^{\frac{n}{n-2}}}}dx \\
&\leq &\int_{B_{R}}e^{a_{2,n}\frac{\left\vert v\right\vert ^{\frac{n}{n-2}}}{%
\left\Vert \Delta v\right\Vert _{L^{\frac{n}{2}}\left( B_{R}\right) }^{\frac{%
n}{n-2}}}}dx \\
&\leq &c\left( n\right) R^{n}.
\end{eqnarray*}
\end{proof}

\begin{lemma}
\label{lem3.4}Let $u\in W^{2,\frac{n}{2}}\left( B_{R}^{+}\right) $ and $a>0$%
, then%
\begin{equation}
\int_{B_{R}^{+}}e^{a\left\vert u\right\vert ^{\frac{n}{n-2}}}dx<\infty .
\label{eq3.31}
\end{equation}
\end{lemma}

\begin{proof}
We can find $\widetilde{u}\in W^{2,\frac{n}{2}}\left( B_{R}\right) $ such
that $\left. \widetilde{u}\right\vert _{B_{R}^{+}}=u$. Then it follows from
Lemma \ref{lem3.1} that%
\begin{equation*}
\int_{B_{R}^{+}}e^{a\left\vert u\right\vert ^{\frac{n}{n-2}}}dx\leq
\int_{B_{R}}e^{a\left\vert \widetilde{u}\right\vert ^{\frac{n}{n-2}%
}}dx<\infty .
\end{equation*}
\end{proof}

\begin{proposition}
\label{prop3.2}Let $0<R\leq 1$, $g$ be a $C^{1}$ Riemannian metric on $%
\overline{B_{R}^{+}}$ such that $g_{ij}\left( 0\right) =\delta _{ij}$ for $%
1\leq i,j\leq n$. Assume $u_{i}\in W^{2,\frac{n}{2}}\left( B_{R}^{+}\right) $%
, $\partial _{n}u_{i}=0$ on $\Sigma _{R}$, $u_{i}\rightharpoonup u$ weakly
in $W^{2,\frac{n}{2}}\left( B_{R}^{+}\right) $ and%
\begin{equation*}
\left\vert \Delta _{g}u_{i}\right\vert ^{\frac{n}{2}}d\mu \rightarrow
\left\vert \Delta _{g}u\right\vert ^{\frac{n}{2}}d\mu +\sigma \text{ as
measure on }B_{R}^{+}\text{.}
\end{equation*}%
If $0<p<\sigma \left( \left\{ 0\right\} \right) ^{-\frac{2}{n-2}}$, then
there exists $r>0$ such that%
\begin{equation}
\sup_{i}\int_{B_{r}^{+}}e^{2^{-\frac{2}{n-2}}a_{2,n}p\left\vert
u_{i}\right\vert ^{\frac{n}{n-2}}}d\mu <\infty .  \label{eq3.32}
\end{equation}%
Here%
\begin{equation}
a_{2,n}=\frac{n}{\left\vert \mathbb{S}^{n-1}\right\vert }\left( \frac{4\pi ^{%
\frac{n}{2}}}{\Gamma \left( \frac{n-2}{2}\right) }\right) ^{\frac{n}{n-2}}.
\label{eq3.33}
\end{equation}
\end{proposition}

First we observe that if $v\in W^{2,\frac{n}{2}}\left( B_{R}^{+}\right) $, $%
v\left( x\right) =0$ for $\left\vert x\right\vert $ close to $R$, $\left.
\partial _{n}v\right\vert _{\Sigma _{R}}=0$, then we have%
\begin{equation}
\left\Vert v\right\Vert _{W^{2,\frac{n}{2}}\left( B_{R}^{+}\right) }\leq
c\left\Vert \Delta v\right\Vert _{L^{\frac{n}{2}}\left( B_{R}^{+}\right) }.
\label{eq3.34}
\end{equation}%
Indeed, let%
\begin{equation*}
\widetilde{v}\left( x\right) =\left\{ 
\begin{array}{cc}
v\left( x\right) , & \text{if }x_{n}>0; \\ 
v\left( x^{\prime },-x_{n}\right) , & \text{if }x_{n}<0%
\end{array}%
\right.
\end{equation*}%
for $x\in B_{R}$, then $\widetilde{v}\in W_{0}^{2,\frac{n}{2}}\left(
B_{R}\right) $. Hence elliptic estimates gives us%
\begin{equation*}
\left\Vert v\right\Vert _{W^{2,\frac{n}{2}}\left( B_{R}^{+}\right) }\leq
\left\Vert \widetilde{v}\right\Vert _{W^{2,\frac{n}{2}}\left( B_{R}\right)
}\leq c\left\Vert \Delta \widetilde{v}\right\Vert _{L^{\frac{n}{2}}\left(
B_{R}\right) }\leq c\left\Vert \Delta v\right\Vert _{L^{\frac{n}{2}}\left(
B_{R}^{+}\right) }.
\end{equation*}

To continue, we recall (\ref{eq3.10}) that if $0<R_{1}<R$ is a small number
and $v\in W^{2,\frac{n}{2}}\left( B_{R_{1}}^{+}\right) $, $v\left( x\right)
=0$ for $\left\vert x\right\vert $ close to $R_{1}$, $\left. \partial
_{n}v\right\vert _{\Sigma _{R_{1}}}=0$, then%
\begin{equation*}
\left\vert \Delta v\right\vert \leq \left\vert \Delta _{g}v\right\vert
+\varepsilon _{1}\left\vert D^{2}v\right\vert +c\left\vert \nabla
v\right\vert ,
\end{equation*}%
$\varepsilon _{1}=\varepsilon _{1}\left( R_{1}\right) >0$ with $\varepsilon
_{1}\rightarrow 0$ as $R_{1}\rightarrow 0^{+}$. Hence%
\begin{eqnarray*}
&&\left\Vert \Delta v\right\Vert _{L^{\frac{n}{2}}\left( B_{R}^{+}\right) }^{%
\frac{n}{2}} \\
&\leq &\left( \left\Vert \Delta _{g}v\right\Vert _{L^{\frac{n}{2}%
}}+\varepsilon _{1}\left\Vert D^{2}v\right\Vert _{L^{\frac{n}{2}%
}}+c\left\Vert \nabla v\right\Vert _{L^{\frac{n}{2}}}\right) ^{\frac{n}{2}}
\\
&\leq &\left( 1+\frac{\varepsilon }{16}\right) \left\Vert \Delta
_{g}v\right\Vert _{L^{\frac{n}{2}}}^{\frac{n}{2}}+c\left( \varepsilon
\right) \varepsilon _{1}^{\frac{n}{2}}\left\Vert D^{2}v\right\Vert _{L^{%
\frac{n}{2}}}^{\frac{n}{2}}+c\left( \varepsilon \right) \left\Vert \nabla
v\right\Vert _{L^{\frac{n}{2}}}^{\frac{n}{2}} \\
&\leq &\left( 1+\frac{\varepsilon }{16}\right) \left\Vert \Delta
_{g}v\right\Vert _{L^{\frac{n}{2}}}^{\frac{n}{2}}+c\left( \varepsilon
\right) \varepsilon _{1}^{\frac{n}{2}}\left\Vert \Delta v\right\Vert _{L^{%
\frac{n}{2}}}^{\frac{n}{2}}+c\left( \varepsilon \right) \left\Vert \nabla
v\right\Vert _{L^{\frac{n}{2}}}^{\frac{n}{2}}.
\end{eqnarray*}%
It follows that%
\begin{eqnarray*}
\left\Vert \Delta v\right\Vert _{L^{\frac{n}{2}}\left( B_{R}^{+}\right) }^{%
\frac{n}{2}} &\leq &\frac{1+\frac{\varepsilon }{16}}{1-c\left( \varepsilon
\right) \varepsilon _{1}^{\frac{n}{2}}}\left\Vert \Delta _{g}v\right\Vert
_{L^{\frac{n}{2}}}^{\frac{n}{2}}+\frac{c\left( \varepsilon \right) }{%
1-c\left( \varepsilon \right) \varepsilon _{1}^{\frac{n}{2}}}\left\Vert
\nabla v\right\Vert _{L^{\frac{n}{2}}}^{\frac{n}{2}} \\
&\leq &\left( 1+\frac{\varepsilon }{8}\right) \left\Vert \Delta
_{g}v\right\Vert _{L^{\frac{n}{2}}}^{\frac{n}{2}}+c\left( \varepsilon
\right) \left\Vert \nabla v\right\Vert _{L^{\frac{n}{2}}}^{\frac{n}{2}}
\end{eqnarray*}%
if $R_{1}$ is small enough.

\begin{proof}[Proof of Proposition \protect\ref{prop3.2}]
Fix $p_{1}\in \left( p,\sigma \left( \left\{ 0\right\} \right) ^{-\frac{2}{%
n-2}}\right) $, then%
\begin{equation}
\sigma \left( \left\{ 0\right\} \right) <\frac{1}{p_{1}^{\frac{n-2}{2}}}.
\label{eq3.35}
\end{equation}%
We can find $\varepsilon >0$ s.t.%
\begin{equation}
\left( 1+\varepsilon \right) \sigma \left( \left\{ 0\right\} \right) <\frac{1%
}{p_{1}^{\frac{n-2}{2}}}  \label{eq3.36}
\end{equation}%
and%
\begin{equation}
\left( 1+\varepsilon \right) p<p_{1}.  \label{eq3.37}
\end{equation}%
Let $v_{i}=u_{i}-u$, then $v_{i}\rightharpoonup 0$ weakly in $W^{2,\frac{n}{2%
}}\left( B_{R}^{+}\right) $ and $v_{i}\rightarrow 0$ in $W^{1,\frac{n}{2}%
}\left( B_{R}^{+}\right) $. Let $0<R_{1}<R$ be small enough. For radial
symmetric $\varphi \in C_{c}^{\infty }\left( B_{R_{1}}\right) $, we have%
\begin{eqnarray*}
&&\left\Vert \Delta \left( \varphi v_{i}\right) \right\Vert _{L^{\frac{n}{2}%
}\left( B_{R}^{+}\right) }^{\frac{n}{2}} \\
&\leq &\left( 1+\frac{\varepsilon }{8}\right) \left\Vert \Delta _{g}\left(
\varphi v_{i}\right) \right\Vert _{L^{\frac{n}{2}}}^{\frac{n}{2}}+c\left(
\varepsilon \right) \left\Vert \nabla \left( \varphi v_{i}\right)
\right\Vert _{L^{\frac{n}{2}}}^{\frac{n}{2}} \\
&\leq &\left( 1+\frac{\varepsilon }{8}\right) \left( \left\Vert \varphi
\Delta _{g}v_{i}\right\Vert _{L^{\frac{n}{2}}}+c\left\Vert \varphi
\right\Vert _{W^{2,\infty }}\left\Vert v_{i}\right\Vert _{W^{1,\frac{n}{2}%
}}\right) ^{\frac{n}{2}}+c\left( \varepsilon \right) \left\Vert \nabla
\left( \varphi v_{i}\right) \right\Vert _{L^{\frac{n}{2}}}^{\frac{n}{2}} \\
&\leq &\left( 1+\frac{\varepsilon }{4}\right) \left\Vert \varphi \Delta
_{g}v_{i}\right\Vert _{L^{\frac{n}{2}}}^{\frac{n}{2}}+c\left( \varepsilon
\right) \left\Vert \varphi \right\Vert _{W^{2,\infty }}^{\frac{n}{2}%
}\left\Vert v_{i}\right\Vert _{W^{1,\frac{n}{2}}}^{\frac{n}{2}} \\
&\leq &\left( 1+\frac{\varepsilon }{4}\right) \left( \left\Vert \varphi
\Delta _{g}u_{i}\right\Vert _{L^{\frac{n}{2}}}+\left\Vert \varphi \Delta
_{g}u\right\Vert _{L^{\frac{n}{2}}}\right) ^{\frac{n}{2}}+c\left(
\varepsilon \right) \left\Vert \varphi \right\Vert _{W^{2,\infty }}^{\frac{n%
}{2}}\left\Vert v_{i}\right\Vert _{W^{1,\frac{n}{2}}}^{\frac{n}{2}} \\
&\leq &\left( 1+\frac{\varepsilon }{2}\right) \left\Vert \varphi \Delta
_{g}u_{i}\right\Vert _{L^{\frac{n}{2}}}^{\frac{n}{2}}+c\left( \varepsilon
\right) \left\Vert \varphi \Delta _{g}u\right\Vert _{L^{\frac{n}{2}}}^{\frac{%
n}{2}}+c\left( \varepsilon \right) \left\Vert \varphi \right\Vert
_{W^{2,\infty }}^{\frac{n}{2}}\left\Vert v_{i}\right\Vert _{W^{1,\frac{n}{2}%
}}^{\frac{n}{2}} \\
&\leq &\left( 1+\varepsilon \right) \left\Vert \varphi \Delta
_{g}u_{i}\right\Vert _{L^{\frac{n}{2}}\left( B_{R}^{+},d\mu \right) }^{\frac{%
n}{2}}+c\left( \varepsilon \right) \left\Vert \varphi \Delta
_{g}u\right\Vert _{L^{\frac{n}{2}}\left( B_{R}^{+},d\mu \right) }^{\frac{n}{2%
}}+c\left( \varepsilon \right) \left\Vert \varphi \right\Vert _{W^{2,\infty
}}^{\frac{n}{2}}\left\Vert v_{i}\right\Vert _{W^{1,\frac{n}{2}}}^{\frac{n}{2}%
}.
\end{eqnarray*}%
Hence%
\begin{eqnarray*}
&&\lim \sup_{i\rightarrow \infty }\left\Vert \Delta \left( \varphi
v_{i}\right) \right\Vert _{L^{\frac{n}{2}}\left( B_{R}^{+}\right) }^{\frac{n%
}{2}} \\
&\leq &\left( 1+\varepsilon \right) \int_{B_{R}^{+}}\left\vert \varphi
\right\vert ^{\frac{n}{2}}d\sigma +\left( 1+\varepsilon \right)
\int_{B_{R}^{+}}\left\vert \varphi \right\vert ^{\frac{n}{2}}\left\vert
\Delta _{g}u\right\vert ^{\frac{n}{2}}d\mu +c\left( \varepsilon \right)
\left\Vert \varphi \Delta _{g}u\right\Vert _{L^{\frac{n}{2}}\left(
B_{R}^{+},d\mu \right) }^{\frac{n}{2}} \\
&\leq &\left( 1+\varepsilon \right) \int_{B_{R}^{+}}\left\vert \varphi
\right\vert ^{\frac{n}{2}}d\sigma +c\left( \varepsilon \right)
\int_{B_{R}^{+}}\left\vert \varphi \right\vert ^{\frac{n}{2}}\left\vert
\Delta _{g}u\right\vert ^{\frac{n}{2}}d\mu .
\end{eqnarray*}%
Since $\left( 1+\varepsilon \right) \sigma \left( \left\{ 0\right\} \right) <%
\frac{1}{p_{1}^{\frac{n-2}{2}}}$, we can choose a radial function $\varphi
\in C_{c}^{\infty }\left( B_{R_{1}}\right) $ such that $\left. \varphi
\right\vert _{B_{r}}=1$ for some $r>0$ and%
\begin{equation}
\left( 1+\varepsilon \right) \int_{B_{R}^{+}}\left\vert \varphi \right\vert
^{\frac{n}{2}}d\sigma +c\left( \varepsilon \right)
\int_{B_{R}^{+}}\left\vert \varphi \right\vert ^{\frac{n}{2}}\left\vert
\Delta _{g}u\right\vert ^{\frac{n}{2}}d\mu <\frac{1}{p_{1}^{\frac{n-2}{2}}}.
\label{eq3.38}
\end{equation}%
Hence for $i$ large enough, we have%
\begin{equation}
\left\Vert \Delta \left( \varphi v_{i}\right) \right\Vert _{L^{\frac{n}{2}%
}\left( B_{R}^{+}\right) }^{\frac{n}{2}}<\frac{1}{p_{1}^{\frac{n-2}{2}}}.
\label{eq3.39}
\end{equation}%
This implies%
\begin{equation}
\left\Vert \Delta \left( \varphi v_{i}\right) \right\Vert _{L^{\frac{n}{2}%
}\left( B_{R}^{+}\right) }^{\frac{n}{n-2}}<\frac{1}{p_{1}}.  \label{eq3.40}
\end{equation}%
On the other hand,%
\begin{equation*}
\left\vert u_{i}\right\vert ^{\frac{n}{n-2}}\leq \left( \left\vert
v_{i}\right\vert +\left\vert u\right\vert \right) ^{\frac{n}{n-2}}\leq
\left( 1+\varepsilon \right) \left\vert v_{i}\right\vert ^{\frac{n}{n-2}%
}+c\left( \varepsilon \right) \left\vert u\right\vert ^{\frac{n}{n-2}},
\end{equation*}%
hence%
\begin{equation*}
e^{2^{-\frac{2}{n-2}}a_{2,n}\left\vert u_{i}\right\vert ^{\frac{n}{n-2}%
}}\leq e^{\left( 1+\varepsilon \right) 2^{-\frac{2}{n-2}}a_{2,n}\left\vert
v_{i}\right\vert ^{\frac{n}{n-2}}}e^{c\left( \varepsilon \right) \left\vert
u\right\vert ^{\frac{n}{n-2}}}.
\end{equation*}%
It follows from Lemma \ref{lem3.3} that%
\begin{eqnarray*}
\int_{B_{r}^{+}}e^{2^{-\frac{2}{n-2}}a_{2,n}p_{1}\left\vert v_{i}\right\vert
^{\frac{n}{n-2}}}d\mu &\leq &c\int_{B_{R}^{+}}e^{2^{-\frac{2}{n-2}%
}a_{2,n}p_{1}\left\vert \varphi v_{i}\right\vert ^{\frac{n}{n-2}}}dx \\
&\leq &c\int_{B_{R}^{+}}e^{2^{-\frac{2}{n-2}}a_{2,n}\frac{\left\vert \varphi
v_{i}\right\vert ^{\frac{n}{n-2}}}{\left\Vert \nabla \left( \varphi
v_{i}\right) \right\Vert _{L^{n}}^{\frac{n}{n-2}}}}dx \\
&\leq &c.
\end{eqnarray*}%
Using $p<\frac{p_{1}}{1+\varepsilon }$ and Lemma \ref{lem3.4}, it follows
from Holder's inequality that $e^{2^{-\frac{2}{n-2}}a_{2,n}\left\vert
u_{i}\right\vert ^{\frac{n}{n-2}}}$ is bounded in $L^{p}\left(
B_{r}^{+},d\mu \right) $.
\end{proof}

\begin{theorem}
\label{thm3.2}Let $\left( M^{n},g\right) $ be a $C^{4}$ compact Riemannian
manifold with boundary and $g$ be a $C^{3}$ Riemannian metric on $M$. If $%
E\subset M$ with $\mu \left( E\right) \geq \delta >0$, $u\in W^{2,\frac{n}{2}%
}\left( M\right) \backslash \left\{ 0\right\} $ with $u_{E}=0$ and $\left. 
\frac{\partial u}{\partial \nu }\right\vert _{\partial M}=0$ (here $\nu $ is
the unit outer normal direction on $\partial M$), then for any $0<a<a_{2,n}=%
\frac{n}{\left\vert \mathbb{S}^{n-1}\right\vert }\left( \frac{4\pi ^{\frac{n%
}{2}}}{\Gamma \left( \frac{n-2}{2}\right) }\right) ^{\frac{n}{n-2}}$,%
\begin{equation}
\int_{M}e^{2^{-\frac{2}{n-2}}a\frac{\left\vert u\right\vert ^{\frac{n}{n-2}}%
}{\left\Vert \Delta u\right\Vert _{L^{\frac{n}{2}}}^{\frac{n}{n-2}}}}d\mu
\leq c\left( a,\delta \right) <\infty .  \label{eq3.41}
\end{equation}
\end{theorem}

\begin{proof}
For any $v\in W^{2,\frac{n}{2}}\left( M\right) $ with $\left. \frac{\partial
v}{\partial \nu }\right\vert _{\partial M}=0$, standard elliptic theory and
compactness argument tells us%
\begin{equation*}
\left\Vert v-v_{M}\right\Vert _{L^{\frac{n}{2}}\left( M\right) }\leq
c\left\Vert \Delta v\right\Vert _{L^{\frac{n}{2}}\left( M\right) }.
\end{equation*}%
On the other hand,%
\begin{equation*}
\left\Vert v-v_{E}\right\Vert _{L^{\frac{n}{2}}}\leq \frac{c}{\mu \left(
E\right) ^{\frac{2}{n}}}\left\Vert v\right\Vert _{L^{\frac{n}{2}}}\leq
c\left( \delta \right) \left\Vert v\right\Vert _{L^{\frac{n}{2}}}.
\end{equation*}%
Replacing $v$ by $v-v_{M}$, we see%
\begin{equation*}
\left\Vert v-v_{E}\right\Vert _{L^{\frac{n}{2}}}\leq c\left( \delta \right)
\left\Vert v-v_{M}\right\Vert _{L^{\frac{n}{2}}}\leq c\left( \delta \right)
\left\Vert \Delta v\right\Vert _{L^{\frac{n}{2}}}.
\end{equation*}

If (\ref{eq3.41}) is not true, then we can find a sequence $u_{i}\in W^{2,%
\frac{n}{2}}\left( M\right) $, $E_{i}\subset M$ with $\mu \left(
E_{i}\right) \geq \delta $, $u_{i,E_{i}}=0$, $\left\Vert \Delta
u_{i}\right\Vert _{L^{\frac{n}{2}}}=1$, $\left. \frac{\partial u_{i}}{%
\partial \nu }\right\vert _{\partial M}=0$ and%
\begin{equation*}
\int_{M}e^{2^{-\frac{2}{n-2}}a\left\vert u_{i}\right\vert ^{\frac{n}{n-2}%
}}d\mu \rightarrow \infty
\end{equation*}%
as $i\rightarrow \infty $. Since%
\begin{equation*}
\left\Vert u_{i}\right\Vert _{L^{\frac{n}{2}}}=\left\Vert
u_{i}-u_{i,E_{i}}\right\Vert _{L^{\frac{n}{2}}}\leq c\left( \delta \right)
\left\Vert \Delta u_{i}\right\Vert _{L^{\frac{n}{2}}}=c\left( \delta \right)
,
\end{equation*}%
standard elliptic estimate gives us%
\begin{equation*}
\left\Vert u_{i}\right\Vert _{W^{2,\frac{n}{2}}}\leq c\left( \left\Vert
\Delta u_{i}\right\Vert _{L^{\frac{n}{2}}}+\left\Vert u_{i}\right\Vert _{L^{%
\frac{n}{2}}}\right) \leq c\left( \delta \right) .
\end{equation*}%
Hence $u_{i}$ is bounded in $W^{2,\frac{n}{2}}\left( M\right) $. After
passing to a subsequence we can find $u\in W^{2,\frac{n}{2}}\left( M\right) $
such that $u_{i}\rightharpoonup u$ weakly in $W^{2,\frac{n}{2}}\left(
M\right) $ and a measure on $M$, $\sigma $ such that%
\begin{equation*}
\left\vert \Delta u_{i}\right\vert ^{\frac{n}{2}}d\mu \rightarrow \left\vert
\Delta u\right\vert ^{\frac{n}{2}}d\mu +\sigma
\end{equation*}%
as measure. Note that $\sigma \left( M\right) \leq 1$. For any $y\in
M\backslash \partial M$, since%
\begin{equation*}
0<\frac{a}{a_{2,n}}<\sigma \left( \left\{ y\right\} \right) ^{-\frac{2}{n-2}%
},
\end{equation*}%
it follows from Proposition \ref{prop3.1} that for some $r>0$, we have%
\begin{equation*}
\sup_{i}\int_{B_{r}\left( y\right) }e^{2^{-\frac{2}{n-2}}a\left\vert
u_{i}\right\vert ^{\frac{n}{n-2}}}d\mu \leq \sup_{i}\int_{B_{r}\left(
y\right) }e^{a\left\vert u_{i}\right\vert ^{\frac{n}{n-2}}}d\mu <\infty .
\end{equation*}%
Next we deal with the case $y\in \partial M$. Note for $\varepsilon _{0}>0$
small enough, we have a well defined map%
\begin{equation*}
\phi :\partial M\times \left[ 0,\varepsilon _{0}\right) \rightarrow M:\left(
\xi ,t\right) \mapsto \exp _{\xi }\left( -t\nu \left( \xi \right) \right) .
\end{equation*}%
By the compactness of $M$, we can find $\varepsilon _{0}>0$ such that $\phi
\left( \partial M\times \left[ 0,\varepsilon _{0}\right) \right) $ is open
and $\phi $ is a $C^{2}$ diffeomorphism. We write $\phi ^{-1}\left( z\right)
=\left( \xi \left( z\right) ,t\left( z\right) \right) $. Let $s_{1},\cdots
,s_{n-1}$ be a coordinate near $y$ on $\partial M$ such that $s_{i}\left(
y\right) =0$ and $\left\langle \partial _{s_{i}},\partial
_{s_{j}}\right\rangle \left( y\right) =\delta _{ij}$ for $1\leq i,j\leq n-1$%
. Then we define a coordinate $x_{1},\cdots ,x_{n}$ near $y$ as $x_{i}\left(
z\right) =s_{i}\left( \xi \left( z\right) \right) $ for $1\leq i\leq n-1$
and $x_{n}\left( z\right) =t\left( z\right) $. It is clear that $\left.
\partial _{x_{n}}\right\vert _{x_{n}=0}=-\nu $, hence we see $\left.
\partial _{n}u_{i}\right\vert _{\Sigma _{R}}=0$. It follows from Proposition %
\ref{prop3.2} that for some $r>0$, we have%
\begin{equation*}
\sup_{i}\int_{B_{r}\left( y\right) }e^{2^{-\frac{2}{n-2}}a\left\vert
u_{i}\right\vert ^{\frac{n}{n-2}}}d\mu <\infty .
\end{equation*}%
A covering argument implies%
\begin{equation*}
\sup_{i}\int_{M}e^{2^{-\frac{2}{n-2}}a\left\vert u_{i}\right\vert ^{\frac{n}{%
n-2}}}d\mu <\infty .
\end{equation*}%
This contradicts with the choice of $u_{i}$.
\end{proof}

\begin{example}
\label{ex3.2}Let $\left( M^{n},g\right) $ be a $C^{4}$ compact Riemannian
manifold with boundary and $g$ be a $C^{3}$ Riemannian metric on $M$. We
define%
\begin{equation}
\kappa _{2,N}\left( M,g\right) =\inf_{\substack{ u\in W^{2,\frac{n}{2}%
}\left( M\right) \backslash \left\{ 0\right\}  \\ u_{M}=0,\left. \frac{%
\partial u}{\partial \nu }\right\vert _{\partial M}=0}}\frac{\left\Vert
\Delta u\right\Vert _{L^{\frac{n}{2}}}}{\left\Vert u\right\Vert _{L^{\frac{n%
}{2}}}},  \label{eq3.42}
\end{equation}%
here $\nu $ is the unit outnormal direction on $\partial M$. Assume $0\leq
\kappa <\kappa _{2,N}\left( M,g\right) $, $0<a<a_{2,n}$, $u\in W^{2,\frac{n}{%
2}}\left( M\right) $ with $u_{M}=0$,$\left. \frac{\partial u}{\partial \nu }%
\right\vert _{\partial M}=0$ and%
\begin{equation}
\left\Vert \Delta u\right\Vert _{L^{\frac{n}{2}}}^{\frac{n}{2}}-\kappa ^{%
\frac{n}{2}}\left\Vert u\right\Vert _{L^{\frac{n}{2}}}^{\frac{n}{2}}\leq 1,
\label{eq3.43}
\end{equation}%
then%
\begin{equation}
\int_{M}e^{2^{-\frac{2}{n-2}}a\left\vert u\right\vert ^{\frac{n}{n-2}}}d\mu
\leq c\left( \kappa ,a\right) <\infty .  \label{eq3.44}
\end{equation}
\end{example}

Since the proof of Example \ref{ex3.2} is almost identical to the proof of
Example \ref{ex1.1} (using Proposition \ref{prop3.1} and \ref{prop3.2} when
necessary), we omit it here.

Next we move to the second boundary condition.

\begin{lemma}
\label{lem3.5}Assume $u\in W^{2,\frac{n}{2}}\left( B_{R}^{+}\right)
\backslash \left\{ 0\right\} $ such that $u\left( x\right) =0$ for $%
\left\vert x\right\vert $ close to $R$ and $\left. u\right\vert _{\Sigma
_{R}}=0$, then%
\begin{equation}
\int_{B_{R}^{+}}e^{2^{-\frac{2}{n-2}}a_{2,n}\frac{\left\vert u\right\vert ^{%
\frac{n}{n-2}}}{\left\Vert \Delta u\right\Vert _{L^{\frac{n}{2}}\left(
B_{R}^{+}\right) }^{\frac{n}{n-2}}}}dx\leq c\left( n\right) R^{n}.
\label{eq3.45}
\end{equation}
\end{lemma}

\begin{proof}
For $\left\vert x\right\vert <R$, we define%
\begin{equation*}
v\left( x\right) =\left\{ 
\begin{array}{cc}
u\left( x\right) , & \text{if }x_{n}>0; \\ 
-u\left( x^{\prime },-x_{n}\right) , & \text{if }x_{n}<0\text{.}%
\end{array}%
\right.
\end{equation*}%
Then $v\in W_{0}^{2,\frac{n}{2}}\left( B_{R}\right) $ with $\left\Vert
\Delta v\right\Vert _{L^{\frac{n}{2}}\left( B_{R}\right) }^{\frac{n}{2}%
}=2\left\Vert \Delta u\right\Vert _{L^{\frac{n}{2}}\left( B_{R}^{+}\right)
}^{\frac{n}{2}}$. Hence%
\begin{eqnarray*}
\int_{B_{R}^{+}}e^{2^{-\frac{2}{n-2}}a_{2,n}\frac{\left\vert u\right\vert ^{%
\frac{n}{n-2}}}{\left\Vert \Delta u\right\Vert _{L^{\frac{n}{2}}\left(
B_{R}^{+}\right) }^{\frac{n}{n-2}}}}dx &=&\int_{B_{R}^{+}}e^{a_{2,n}\frac{%
\left\vert u\right\vert ^{\frac{n}{n-2}}}{\left\Vert \Delta v\right\Vert
_{L^{\frac{n}{2}}\left( B_{R}\right) }^{\frac{n}{n-2}}}}dx \\
&\leq &\int_{B_{R}}e^{a_{2,n}\frac{\left\vert v\right\vert ^{\frac{n}{n-2}}}{%
\left\Vert \Delta v\right\Vert _{L^{\frac{n}{2}}\left( B_{R}\right) }^{\frac{%
n}{n-2}}}}dx \\
&\leq &c\left( n\right) R^{n}.
\end{eqnarray*}
\end{proof}

\begin{proposition}
\label{prop3.3}Let $0<R\leq 1$, $g$ be a $C^{1}$ Riemannian metric on $%
\overline{B_{R}^{+}}$. Assume $u_{i}\in W^{2,\frac{n}{2}}\left(
B_{R}^{+}\right) $, $\left. u_{i}\right\vert _{\Sigma _{R}}=0$, $%
u_{i}\rightharpoonup u$ weakly in $W^{2,\frac{n}{2}}\left( B_{R}^{+}\right) $
and%
\begin{equation*}
\left\vert \Delta _{g}u_{i}\right\vert ^{\frac{n}{2}}d\mu \rightarrow
\left\vert \Delta _{g}u\right\vert ^{\frac{n}{2}}d\mu +\sigma \text{ as
measure on }B_{R}^{+}\text{.}
\end{equation*}%
If $0<p<\sigma \left( \left\{ 0\right\} \right) ^{-\frac{2}{n-2}}$, then
there exists $r>0$ such that%
\begin{equation}
\sup_{i}\int_{B_{r}^{+}}e^{2^{-\frac{2}{n-2}}a_{2,n}p\left\vert
u_{i}\right\vert ^{\frac{n}{n-2}}}d\mu <\infty .  \label{eq3.46}
\end{equation}
\end{proposition}

By a linear changing of variable and shrinking $R$ if necessary we can
assume $g=g_{ij}dx_{i}dx_{j}$ with $g_{ij}\left( 0\right) =\delta _{ij}$.
The remaining proof is almost identical to the proof of Proposition \ref%
{prop3.2} (using Lemma \ref{lem3.5} instead of Lemma \ref{lem3.3}), we omit
it here.

\begin{theorem}
\label{thm3.3}Let $M^{n}$ be a $C^{2}$ compact manifold with boundary and $g$
be a $C^{1}$ Riemannian metric on $M$. If $E\subset M$ with $\mu \left(
E\right) \geq \delta >0$, $u\in W^{2,\frac{n}{2}}\left( M\right) \cap
W_{0}^{1,\frac{n}{2}}\left( M\right) \backslash \left\{ 0\right\} $ with $%
u_{E}=0$, then for any $0<a<a_{2,n}=\frac{n}{\left\vert \mathbb{S}%
^{n-1}\right\vert }\left( \frac{4\pi ^{\frac{n}{2}}}{\Gamma \left( \frac{n-2%
}{2}\right) }\right) ^{\frac{n}{n-2}}$,%
\begin{equation}
\int_{M}e^{2^{-\frac{2}{n-2}}a\frac{\left\vert u\right\vert ^{\frac{n}{n-2}}%
}{\left\Vert \Delta u\right\Vert _{L^{\frac{n}{2}}}^{\frac{n}{n-2}}}}d\mu
\leq c\left( a,\delta \right) <\infty .  \label{eq3.47}
\end{equation}
\end{theorem}

The proof is almost identical to the proof of Theorem \ref{thm3.2} (using
Proposition \ref{prop3.3} instead of Proposition \ref{prop3.2}), we omit it
here.

\begin{example}
\label{ex3.3}Let $M^{n}$ be a $C^{2}$ compact manifold with boundary and $g$
be a $C^{1}$ Riemannian metric on $M$. We denote%
\begin{equation}
\kappa _{2,D}\left( M,g\right) =\inf_{u\in W^{2,\frac{n}{2}}\left( M\right)
\cap W_{0}^{1,\frac{n}{2}}\left( M\right) \backslash \left\{ 0\right\} }%
\frac{\left\Vert \Delta u\right\Vert _{L^{\frac{n}{2}}}}{\left\Vert
u\right\Vert _{L^{\frac{n}{2}}}}.  \label{eq3.48}
\end{equation}%
Assume $0\leq \kappa <\kappa _{2,D}\left( M,g\right) $, $0<a<a_{2,n}$, $u\in
W^{2,\frac{n}{2}}\left( M\right) \cap W_{0}^{1,\frac{n}{2}}\left( M\right) $
and%
\begin{equation}
\left\Vert \Delta u\right\Vert _{L^{\frac{n}{2}}}^{\frac{n}{2}}-\kappa ^{%
\frac{n}{2}}\left\Vert u\right\Vert _{L^{\frac{n}{2}}}^{\frac{n}{2}}\leq 1,
\label{eq3.49}
\end{equation}%
then%
\begin{equation}
\int_{M}e^{2^{-\frac{2}{n-2}}a\left\vert u\right\vert ^{\frac{n}{n-2}}}d\mu
\leq c\left( \kappa ,a\right) <\infty .  \label{eq3.50}
\end{equation}
\end{example}

Since the proof of Example \ref{ex3.3} is almost identical to the proof of
Example \ref{ex1.1} (using Proposition \ref{prop3.1} and \ref{prop3.3} when
necessary), we omit it here.

At last we turn to the third boundary condition.

\begin{proposition}
\label{prop3.4}Let $0<R\leq 1$, $g$ be a $C^{1}$ Riemannian metric on $%
\overline{B_{R}^{+}}$. Assume $u_{i}\in W^{2,\frac{n}{2}}\left(
B_{R}^{+}\right) $, $\left. u\right\vert _{\Sigma _{R}}=0$, $\left. \partial
_{n}u\right\vert _{\Sigma _{R}}=0$, $u_{i}\rightharpoonup u$ weakly in $W^{2,%
\frac{n}{2}}\left( B_{R}^{+}\right) $ and%
\begin{equation*}
\left\vert \Delta _{g}u_{i}\right\vert ^{\frac{n}{2}}d\mu \rightarrow
\left\vert \Delta _{g}u\right\vert ^{\frac{n}{2}}d\mu +\sigma \text{ as
measure on }B_{R}^{+}\text{.}
\end{equation*}%
If $0<p<\sigma \left( \left\{ 0\right\} \right) ^{-\frac{2}{n-2}}$, then
there exists $r>0$ such that%
\begin{equation}
\sup_{i}\int_{B_{r}^{+}}e^{a_{2,n}p\left\vert u_{i}\right\vert ^{\frac{n}{n-2%
}}}d\mu <\infty .  \label{eq3.51}
\end{equation}
\end{proposition}

\begin{proof}
Let $h$ be a $C^{1}$ Riemannian metric on $\overline{B_{R}}$, which is an
extension of $g$. We define%
\begin{equation*}
v_{i}\left( x\right) =\left\{ 
\begin{tabular}{ll}
$u_{i}\left( x\right) ,$ & if $x\in B_{R}^{+}$ \\ 
$0,$ & if $x\in B_{R}\backslash B_{R}^{+}$%
\end{tabular}%
\right. ,\quad v\left( x\right) =\left\{ 
\begin{tabular}{ll}
$u\left( x\right) ,$ & if $x\in B_{R}^{+}$ \\ 
$0,$ & if $x\in B_{R}\backslash B_{R}^{+}$%
\end{tabular}%
\right. ,
\end{equation*}%
and a measure $\tau $ on $B_{R}$ by $\tau \left( E\right) =\sigma \left(
E\cap B_{R}\right) $ for any Borel set $E\subset B_{R}$. Then $v_{i},v\in
W^{2,\frac{n}{2}}\left( B_{R}\right) $, $v_{i}\rightharpoonup v$ weakly in $%
W^{2,\frac{n}{2}}\left( B_{R}\right) $ and%
\begin{equation*}
\left\vert \nabla _{h}v_{i}\right\vert _{h}^{n}d\mu _{h}\rightarrow
\left\vert \nabla _{h}v\right\vert _{h}^{n}d\mu _{h}+\tau \text{ as measure
on }B_{R}.
\end{equation*}%
Here $\mu _{h}$ is the measure associated with $h$. Since $\tau \left(
\left\{ 0\right\} \right) =\sigma \left( \left\{ 0\right\} \right) $, it
follows from Proposition \ref{prop3.1} that for some $r>0$,%
\begin{equation*}
\sup_{i}\int_{B_{r}}e^{a_{2,n}p\left\vert v_{i}\right\vert ^{\frac{n}{n-2}%
}}d\mu _{h}<\infty .
\end{equation*}%
Hence%
\begin{equation*}
\sup_{i}\int_{B_{r}^{+}}e^{a_{2,n}p\left\vert u_{i}\right\vert ^{\frac{n}{n-2%
}}}d\mu <\infty .
\end{equation*}
\end{proof}

\begin{theorem}
\label{thm3.4}Let $M^{n}$ be a $C^{2}$ compact manifold with boundary and $g$
be a $C^{1}$ Riemannian metric on $M$. We denote%
\begin{equation}
\kappa _{2}\left( M,g\right) =\inf_{u\in W_{0}^{2,\frac{n}{2}}\left(
M\right) \backslash \left\{ 0\right\} }\frac{\left\Vert \Delta u\right\Vert
_{L^{\frac{n}{2}}}}{\left\Vert u\right\Vert _{L^{\frac{n}{2}}}}.
\label{eq3.52}
\end{equation}%
Assume $0\leq \kappa <\kappa _{2}\left( M,g\right) $, $0<a<a_{2,n}$, $u\in
W_{0}^{2,\frac{n}{2}}\left( M\right) $ and%
\begin{equation}
\left\Vert \Delta u\right\Vert _{L^{\frac{n}{2}}}^{\frac{n}{2}}-\kappa ^{%
\frac{n}{2}}\left\Vert u\right\Vert _{L^{\frac{n}{2}}}^{\frac{n}{2}}\leq 1,
\label{eq3.53}
\end{equation}%
then%
\begin{equation}
\int_{M}e^{a\left\vert u\right\vert ^{\frac{n}{n-2}}}d\mu \leq c\left(
\kappa ,a\right) <\infty .  \label{eq3.54}
\end{equation}
\end{theorem}

Since the proof of Theorem \ref{thm2.2} is almost identical to the proof of
Example \ref{ex1.1} (using Proposition \ref{prop3.1} and \ref{prop3.4} when
necessary), we omit it here. Theorem \ref{thm3.4} should be compared to \cite%
{LY}.

\section{Higher order Sobolev spaces\label{sec4}}

In this section, we will apply our approach to higher order Sobolev spaces.
Since the arguments are similar to those for first and second order Sobolev
spaces, we omit all the proofs here. For future references we list the
theorems.

\subsection{$W^{m,\frac{n}{m}}\left( M^{n}\right) $ for even $m$}

Let $m\in \mathbb{N}$ be an even number strictly less than $n$. For $R>0$, $%
u\in W_{0}^{m,\frac{n}{m}}\left( B_{R}\right) \backslash \left\{ 0\right\} $%
, it is shown in \cite{A} that%
\begin{equation}
\int_{B_{R}}\exp \left( a_{m,n}\frac{\left\vert u\right\vert ^{\frac{n}{n-m}}%
}{\left\Vert \Delta ^{\frac{m}{2}}u\right\Vert _{L^{\frac{n}{m}}}^{\frac{n}{%
n-m}}}\right) dx\leq c\left( m,n\right) R^{n}.  \label{eq4.1}
\end{equation}%
Here%
\begin{equation}
a_{m,n}=\frac{n}{\left\vert \mathbb{S}^{n-1}\right\vert }\left( \frac{\pi ^{%
\frac{n}{2}}2^{m}\Gamma \left( \frac{m}{2}\right) }{\Gamma \left( \frac{n-m}{%
2}\right) }\right) ^{\frac{n}{n-m}}.  \label{eq4.2}
\end{equation}

\begin{lemma}
\label{lem4.1}If $u\in W^{m,\frac{n}{m}}\left( B_{R}^{n}\right) $, then for
any $a>0$,%
\begin{equation}
\int_{B_{R}}e^{a\left\vert u\right\vert ^{\frac{n}{n-m}}}dx<\infty .
\label{eq4.3}
\end{equation}
\end{lemma}

\begin{proposition}
\label{prop4.1}Let $0<R\leq 1$, $g$ be a $C^{m-1}$ Riemannian metric on $%
\overline{B_{R}^{n}}$. Assume $u_{i}\in W^{m,\frac{n}{m}}\left(
B_{R}^{n}\right) $, $u_{i}\rightharpoonup u$ weakly in $W^{m,\frac{n}{m}%
}\left( B_{R}\right) $ and%
\begin{equation*}
\left\vert \Delta _{g}^{\frac{m}{2}}u_{i}\right\vert ^{\frac{n}{m}}d\mu
\rightarrow \left\vert \Delta _{g}^{\frac{m}{2}}u\right\vert ^{\frac{n}{m}%
}d\mu +\sigma \text{ as measure on }B_{R},
\end{equation*}%
here $\mu $ is the measure associated with the metric $g$. If $0<p<\sigma
\left( \left\{ 0\right\} \right) ^{-\frac{m}{n-m}}$, then for some $r>0$,%
\begin{equation}
\sup_{i}\int_{B_{r}}e^{a_{m,n}p\left\vert u_{i}\right\vert ^{\frac{n}{n-m}%
}}d\mu <\infty .  \label{eq4.4}
\end{equation}
\end{proposition}

\begin{theorem}
\label{thm4.1}Let $M^{n}$ be a $C^{m}$ compact manifold with a $C^{m-1}$
Riemannian metric $g$. If $E\subset M$ with $\mu \left( E\right) \geq \delta
>0$, $u\in W^{m,\frac{n}{m}}\left( M\right) \backslash \left\{ 0\right\} $
with $u_{E}=0$ (here $u_{E}=\frac{1}{\mu \left( E\right) }\int_{E}ud\mu $),
then for any $0<a<a_{m,n}$,%
\begin{equation}
\int_{M}e^{a\frac{\left\vert u\right\vert ^{\frac{n}{n-m}}}{\left\Vert
\Delta ^{\frac{m}{2}}u\right\Vert _{L^{\frac{n}{m}}}^{\frac{n}{n-m}}}}d\mu
\leq c\left( a,\delta \right) <\infty .  \label{eq4.5}
\end{equation}
\end{theorem}

\subsubsection{Functions on manifolds with nonempty boundary}

\begin{lemma}
\label{lem4.2}Assume $u\in W^{m,\frac{n}{m}}\left( B_{R}^{+}\right)
\backslash \left\{ 0\right\} $ such that $u\left( x\right) =0$ for $%
\left\vert x\right\vert $ close to $R$ and $\partial _{n}^{k}u=0$ on $\Sigma
_{R}$ for odd number $k\in \left[ 0,m\right) $, then%
\begin{equation}
\int_{B_{R}^{+}}e^{2^{-\frac{m}{n-m}}a_{m,n}\frac{\left\vert u\right\vert ^{%
\frac{n}{n-m}}}{\left\Vert \Delta ^{\frac{m}{2}}u\right\Vert _{L^{\frac{n}{m}%
}\left( B_{R}^{+}\right) }^{\frac{n}{n-m}}}}dx\leq c\left( n\right) R^{n}.
\label{eq4.6}
\end{equation}
\end{lemma}

This can be done by even extension for $x_{n}$ direction as in the proof of
Lemma \ref{lem3.3}.

\begin{proposition}
\label{prop4.2}Let $0<R\leq 1$, $g$ be a $C^{m-1}$ Riemannian metric on $%
\overline{B_{R}^{+}}$ such that $g_{ij}\left( 0\right) =\delta _{ij}$ for $%
1\leq i,j\leq n$. Assume $u_{i}\in W^{m,\frac{n}{m}}\left( B_{R}^{+}\right) $%
, $\partial _{n}^{k}u=0$ on $\Sigma _{R}$ for odd number $k\in \left[
0,m\right) $, $u_{i}\rightharpoonup u$ weakly in $W^{m,\frac{n}{m}}\left(
B_{R}^{+}\right) $ and%
\begin{equation*}
\left\vert \Delta _{g}^{\frac{m}{2}}u_{i}\right\vert ^{\frac{n}{m}}d\mu
\rightarrow \left\vert \Delta _{g}^{\frac{m}{2}}u\right\vert ^{\frac{n}{m}%
}d\mu +\sigma \text{ as measure on }B_{R}^{+}\text{.}
\end{equation*}%
If $0<p<\sigma \left( \left\{ 0\right\} \right) ^{-\frac{m}{n-m}}$, then
there exists $r>0$ such that%
\begin{equation}
\sup_{i}\int_{B_{r}^{+}}e^{2^{-\frac{m}{n-m}}a_{m,n}p\left\vert
u_{i}\right\vert ^{\frac{n}{n-m}}}d\mu <\infty .  \label{eq4.7}
\end{equation}
\end{proposition}

\begin{theorem}
\label{thm4.2}Let $\left( M^{n},g\right) $ be a $C^{m+2}$ compact Riemannian
manifold with boundary and $g$ be a $C^{m+1}$ Riemannian metric on $M$. If $%
E\subset M$ with $\mu \left( E\right) \geq \delta >0$, $u\in W^{m,\frac{n}{m}%
}\left( M\right) \backslash \left\{ 0\right\} $ with $u_{E}=0$ and $\left.
D^{k}u\left( \overset{k\text{ times}}{\overbrace{\nu ,\cdots ,\nu }}\right)
\right\vert _{\partial M}=0$ for odd number $k\in \left[ 0,m\right) $ (here $%
\nu $ is the unit outer normal direction on $\partial M$), then for any $%
0<a<a_{m,n}$,%
\begin{equation}
\int_{M}e^{2^{-\frac{m}{n-m}}a\frac{\left\vert u\right\vert ^{\frac{n}{n-m}}%
}{\left\Vert \Delta ^{\frac{m}{2}}u\right\Vert _{L^{\frac{n}{m}}}^{\frac{n}{%
n-m}}}}d\mu \leq c\left( a,\delta \right) <\infty .  \label{eq4.8}
\end{equation}
\end{theorem}

\begin{lemma}
\label{lem4.3}Assume $u\in W^{m,\frac{n}{m}}\left( B_{R}^{+}\right)
\backslash \left\{ 0\right\} $ such that $u\left( x\right) =0$ for $%
\left\vert x\right\vert $ close to $R$ and $\partial _{n}^{k}u=0$ on $\Sigma
_{R}$ for even number $k\in \left[ 0,m\right) $, then%
\begin{equation}
\int_{B_{R}^{+}}e^{2^{-\frac{m}{n-m}}a_{m,n}\frac{\left\vert u\right\vert ^{%
\frac{n}{n-m}}}{\left\Vert \Delta ^{\frac{m}{2}}u\right\Vert _{L^{\frac{n}{m}%
}\left( B_{R}^{+}\right) }^{\frac{n}{n-m}}}}dx\leq c\left( n\right) R^{n}.
\label{eq4.9}
\end{equation}
\end{lemma}

This can be done by odd extension in $x_{n}$ direction as in the proof of
Lemma \ref{lem3.5}.

\begin{proposition}
\label{prop4.3}Let $0<R\leq 1$, $g$ be a $C^{m-1}$ Riemannian metric on $%
\overline{B_{R}^{+}}$ such that $g_{ij}\left( 0\right) =\delta _{ij}$ for $%
1\leq i,j\leq n$. Assume $u_{i}\in W^{m,\frac{n}{m}}\left( B_{R}^{+}\right) $%
, $\partial _{n}^{k}u=0$ on $\Sigma _{R}$ for even number $k\in \left[
0,m\right) $, $u_{i}\rightharpoonup u$ weakly in $W^{m,\frac{n}{m}}\left(
B_{R}^{+}\right) $ and%
\begin{equation*}
\left\vert \Delta _{g}^{\frac{m}{2}}u_{i}\right\vert ^{\frac{n}{m}}d\mu
\rightarrow \left\vert \Delta _{g}^{\frac{m}{2}}u\right\vert ^{\frac{n}{m}%
}d\mu +\sigma \text{ as measure on }B_{R}^{+}\text{.}
\end{equation*}%
If $0<p<\sigma \left( \left\{ 0\right\} \right) ^{-\frac{m}{n-m}}$, then
there exists $r>0$ such that%
\begin{equation}
\sup_{i}\int_{B_{r}^{+}}e^{2^{-\frac{m}{n-m}}a_{m,n}p\left\vert
u_{i}\right\vert ^{\frac{n}{n-m}}}d\mu <\infty .  \label{eq4.10}
\end{equation}
\end{proposition}

\begin{theorem}
\label{thm4.3}Let $\left( M^{n},g\right) $ be a $C^{m+2}$ compact Riemannian
manifold with boundary and $g$ be a $C^{m+1}$ Riemannian metric on $M$. If $%
E\subset M$ with $\mu \left( E\right) \geq \delta >0$, $u\in W^{m,\frac{n}{m}%
}\left( M\right) \backslash \left\{ 0\right\} $ with $u_{E}=0$ and $\left.
D^{k}u\left( \overset{k\text{ times}}{\overbrace{\nu ,\cdots ,\nu }}\right)
\right\vert _{\partial M}=0$ for even number $k\in \left[ 0,m\right) $ (here 
$\nu $ is the unit outer normal direction on $\partial M$), then for any $%
0<a<a_{m,n}$,%
\begin{equation}
\int_{M}e^{2^{-\frac{m}{n-m}}a\frac{\left\vert u\right\vert ^{\frac{n}{n-m}}%
}{\left\Vert \Delta ^{\frac{m}{2}}u\right\Vert _{L^{\frac{n}{m}}}^{\frac{n}{%
n-m}}}}d\mu \leq c\left( a,\delta \right) <\infty .  \label{eq4.11}
\end{equation}
\end{theorem}

\begin{proposition}
\label{prop4.4}Let $0<R\leq 1$, $g$ be a $C^{m-1}$ Riemannian metric on $%
\overline{B_{R}^{+}}$. Assume $u_{i}\in W^{m,\frac{n}{m}}\left(
B_{R}^{+}\right) $, $\left. \partial _{n}^{k}u\right\vert _{\Sigma _{R}}=0$
for integer $k\in \left[ 0,m\right) $, $u_{i}\rightharpoonup u$ weakly in $%
W^{m,\frac{n}{m}}\left( B_{R}^{+}\right) $ and%
\begin{equation*}
\left\vert \Delta _{g}^{\frac{m}{2}}u_{i}\right\vert ^{\frac{n}{m}}d\mu
\rightarrow \left\vert \Delta _{g}^{\frac{m}{2}}u\right\vert ^{\frac{n}{m}%
}d\mu +\sigma \text{ as measure on }B_{R}^{+}\text{.}
\end{equation*}%
If $0<p<\sigma \left( \left\{ 0\right\} \right) ^{-\frac{2}{n-2}}$, then
there exists $r>0$ such that%
\begin{equation}
\sup_{i}\int_{B_{r}^{+}}e^{a_{m,n}p\left\vert u_{i}\right\vert ^{\frac{n}{n-m%
}}}d\mu <\infty .  \label{eq4.12}
\end{equation}
\end{proposition}

\begin{theorem}
\label{thm4.4}Let $M^{n}$ be a $C^{m}$ compact manifold with boundary and $g$
be a $C^{m-1}$ Riemannian metric on $M$. We denote%
\begin{equation}
\kappa _{m}\left( M,g\right) =\inf_{u\in W_{0}^{m,\frac{n}{m}}\left(
M\right) \backslash \left\{ 0\right\} }\frac{\left\Vert \Delta ^{\frac{m}{2}%
}u\right\Vert _{L^{\frac{n}{m}}}}{\left\Vert u\right\Vert _{L^{\frac{n}{m}}}}%
.  \label{eq4.13}
\end{equation}%
Assume $0\leq \kappa <\kappa _{m}\left( M,g\right) $, $0<a<a_{m,n}$, $u\in
W_{0}^{m,\frac{n}{m}}\left( M\right) $ and%
\begin{equation}
\left\Vert \Delta u\right\Vert _{L^{\frac{n}{m}}}^{\frac{n}{m}}-\kappa ^{%
\frac{n}{m}}\left\Vert u\right\Vert _{L^{\frac{n}{m}}}^{\frac{n}{m}}\leq 1,
\label{eq4.14}
\end{equation}%
then%
\begin{equation}
\int_{M}e^{a\left\vert u\right\vert ^{\frac{n}{n-m}}}d\mu \leq c\left(
\kappa ,a\right) <\infty .  \label{eq4.15}
\end{equation}
\end{theorem}

\subsection{$W^{m,\frac{n}{m}}\left( M^{n}\right) $ for odd $m$}

Let $m\in \mathbb{N}$ be an odd number strictly less than $n$. For $R>0$, $%
u\in W_{0}^{m,\frac{n}{m}}\left( B_{R}\right) \backslash \left\{ 0\right\} $%
, it is shown in \cite{A} that%
\begin{equation}
\int_{B_{R}}\exp \left( a_{m,n}\frac{\left\vert u\right\vert ^{\frac{n}{n-m}}%
}{\left\Vert \nabla \Delta ^{\frac{m-1}{2}}u\right\Vert _{L^{\frac{n}{m}}}^{%
\frac{n}{n-m}}}\right) dx\leq c\left( m,n\right) R^{n}.  \label{eq4.16}
\end{equation}%
Here%
\begin{equation}
a_{m,n}=\frac{n}{\left\vert \mathbb{S}^{n-1}\right\vert }\left( \frac{\pi ^{%
\frac{n}{2}}2^{m}\Gamma \left( \frac{m+1}{2}\right) }{\Gamma \left( \frac{%
n-m+1}{2}\right) }\right) ^{\frac{n}{n-m}}.  \label{eq4.17}
\end{equation}

With (\ref{eq4.16}) at hands, we can derive similar results for $m$ odd
cases. Indeed if we use $\left\Vert \nabla \Delta ^{\frac{m-1}{2}%
}u\right\Vert _{L^{\frac{n}{m}}}$ instead of $\left\Vert \Delta ^{\frac{m}{2}%
}u\right\Vert _{L^{\frac{n}{m}}}$ (in the $m$ even case), we can get the
corresponding statements. The details are left to interested readers.

\end{document}